\documentclass[12pt,times,a4paper]{amsart}
\usepackage{amssymb,amsmath,amsfonts,latexsym,color,amsthm,tikz}
\usepackage{mathtools}
\usepackage{times}
\usepackage{calc}
\usepackage[T1]{fontenc}
\usepackage[latin2]{inputenc}
\usepackage{graphicx}
\allowdisplaybreaks
\usepackage[all]{xy}
\usepackage{enumitem}
\def\id{\mathop{\rm id}\nolimits}

\def\0D{\Delta^{(0)}}
\def\1D{\Delta^{(1)}}

\newcommand{\longto}{\longrightarrow}

\newtheorem{theorem}{Theorem}[section]
\newtheorem{remark}[theorem]{Remark}
\newtheorem{proposition}[theorem]{Proposition}
\newtheorem{lemma}[theorem]{Lemma}
\newtheorem{corollary}[theorem]{Corollary}

\newtheorem{example}[theorem]{Example}
\newtheorem{definition}[theorem]{Definition}

\def\build#1_#2^#3{\mathrel{
\mathop{\kern 0pt#1}\limits_{#2}^{#3}}}

\numberwithin{equation}{section}

\def\part{\partial}

\def\text{\hbox}

\def\build#1_#2^#3{\mathrel{
\mathop{\kern 0pt#1}\limits_{#2}^{#3}}}

\numberwithin{equation}{section}

\newcommand{\comment}[1]{\relax}

\title[From pushouts of graphs to pullbacks of graph algebras]{\vspace*{-15mm} From length-preserving pushouts of graphs\\
 to one-surjective pullbacks of graph algebras}
\author[P. M.~Hajac]{Piotr M.~Hajac}
\address[P. M.~Hajac]{Instytut Matematyczny, Polska Akademia Nauk, ul. \'Sniadeckich 8, Warszawa, 00-656 Poland}
\email{pmh@impan.pl }

\author[M.~Tobolski]{Mariusz Tobolski}
\address[M.~Tobolski]{Instytut Matematyczny, Uniwersytet Wroc\l{}awski, pl. Grunwaldzki~2, Wroc\l{}aw, 50-384 Poland}
\email{mariusz.tobolski@math.uni.wroc.pl}

\begin{document}
\baselineskip14.5pt
\begin{abstract}%\normalsize
The unions of directed graphs are the simplest examples of pushouts of directed graphs. The conditions under which 
they contravariantly induce surjective gauge-equivariant pullbacks of graph C*-algebras have been well
studied and vastly instantiated in noncommutative topology (e.g., quantum balls and spheres).
Herein, we go beyond the unions of graphs to systematically determine optimal conditions for more 
general length-preserving pushouts 
of graphs under which
they contravariantly induce graded pullbacks of path algebras, Leavitt path algebras, and  graph C*-algebras.
Our pullbacks are surjective only on one side, as dictated by natural examples and K-theory.
The proposed new approach enlarges the scope of applications from 
admissible subgraphs (also called quotient graphs) to generalizations of unlabeled foldings of Stallings and collapsing
the line graphs of graphs to initial graphs. Moreover, we introduce the concept of locally derived graphs, which substantially extends
 the paradigm of derived graphs (or skew products of graphs), and use the projection foldings from locally derived graphs to their base (or voltage) graphs
 to obtain one-surjective pullbacks of graph C*-algebras.
\end{abstract}
\maketitle
\vspace*{-10mm}
{\tableofcontents}
\parskip3.5mm

\section{Introduction}
\noindent
Graph theory is one of the most accessible parts of combinatorics, and one often uses graphs to visualize and study abstract 
mathematical objects. For instance, the structure of a~group can be encoded in its Cayley graph. In the same vein, with every 
(unital, basic, connected) finite-dimensional associative algebra over an algebraically closed field, we can associate a~directed graph 
(or a quiver) from which the algebra in question can be recovered via a path-algebra construction (e.g., see~\cite{ass06}). 
Better still, every finite-dimensional hereditary associative algebra over an algebraically closed field is Morita equivalent to a path 
algebra. These two results show the importance of path algebras in the classification and representation theory of finite-dimensional 
associative algebras. Moreover, Leavitt path algebras, which provide an algebraic backbone of graph C*-algebras, are defined
as quotients of path algebras (e.g., see~\cite{aasm17}).

The construction of path algebras, Leavitt path algebras, and graph C*-algebras can be considered as a functor from a category of 
directed graphs to the category of algebras and algebra homomorphisms in two different ways: covariant and contravariant. The former was explored in
\cite{kr-g09,j-s02,amp07,ht23}. This paper is concerned with the latter. 
The standard category of graphs and graph homomorphism was spectacularly successful in the work of Stallings \cite{s-jr83}, and the 
contravariant induction for admissible subgraphs (also called  quotient graphs)  is ubiquitous, including natural eamples in
 noncommutative topology explored by Hong and Szyma\'nski~\cite{hs02,hs08}. 
However, only considering subgraphs restricts  the standard contravariant functor to
injective graph homomorphisms, which is at odds with an unlabeled Stallings folding (Example~\ref{stallings}),
collapsing the line graph of a Hawaiian earring graph (Equation~\eqref{cuntzfold}), and
shrinking loops (Example~\ref{shrinking}), which all
indentify edges and vertices.

The first  aim of this paper is to unravel optimal conditions for graph homomorphisms to contravariantly induce
graded algebra homomorphisms between path algebras, Leavitt path algebras, and graph C*-algebras. We achieve it in
Lemma~\ref{contrafp}, Theorem~\ref{contralthm}, and Corollary~\ref{corgc*}, respectively, by fine tuning subcategories
of directed graphs. We thus arrive at the category of graphs and admissible graph homomorphisms (see Section~2) as a domain
of a contravariant functor to the category of C*-algebras and $*$-homomorphisms. It turns out that this contravariant functor
is a special case of Katsura's contravariant functor 
from the category of topological graphs and factor maps
to the category of C*-algebras and $*$-homomorphisms.

We introduce and study new types of admissible graph homomorphism.
In particular, as a basic
non-trivial example of a non-injective graph homomorphism contravariantly inducing a gauge-equivariant $*$-homomorphism of graph 
C*-algebras, we have an unlabeled Stallings folding. In this spirit, we define a generalized folding
(Definition~\ref{genfold}) as an example of a non-injective (except in the trivial case) admissible graph homomorphism. 
Better still,  we show that a well-known isomorphism between 
the graph C*-algebra of the line graph of a row-finite
graph without sinks and the graph C*-algebra of the initial graph is contravariantly induced from a generically non-injective
graph homomorphism. Moreover, we significantly extend the concept of derived graphs (which include all Cayley graphs of finite groups) 
by defining locally derived graphs,
and show that, for families of non-trivial finite groups, projection foldings from locally derived graphs to their base  graphs
are non-injective admissible graph homomorphisms. For starters, we exemplify such a graph homomorphism by 
shrinking vertex-simple loops of length $n$ to the loop of length one, 
which induces  inclusions of  the circle C*-algebra $C(S^1)$  in 
$C(S^1)\!\otimes\! M_n(\mathbb{C})$. Then, we construct
$M_2(\mathcal{O}_3)$ as the graph C*-algebra of a locally derived graph of the Hawaiian earring graph
with two loops, and obtain a unital $*$-homomorphism $\mathcal{O}_2\to M_2(\mathcal{O}_3)$.

Furthermore, an unexpected and important application of the contravariant induction was found recently in~\cite{hl24}. Therein, the authors
construct a $U(1)$-equivariant unital $*$-homomorphism $\mathcal{O}_N\to M_k(\mathcal{O}_M)$ whenver $M-1=k(N-1)$, which is a necessary
condition given by K-theory. The construction is given by the contravariant functor applied to a non-injective admissible graph homomorphism.
It clearly exemplifies in action the target-bijectivity condition, which is the pivotal condition of admissibility. This application of contravariant
functoriality complements the application of covariant functoriality \cite{ht23} unravelling the unital \mbox{$*$-homomorphisms} $\mathcal{O}_M\to \mathcal{O}_N$
of Kawamura \cite[Lemma~2.1 and Section~6.1]{k-k09} (cf.~\cite[Section~3.3]{j-c77}) constructed whenever the same congruence $M-1=k(N-1)$, necessitated by K-theory, is satisfied.

The second and principal outcome of this article are pushout-to-pullback theorems: 
for path algebras (Theorem~\ref{pushpullpaththm}),
 Leavitt path algebras (Theorem~\ref{main}), and graph C*-algebras (Theorem~\ref{c*cor}). 
These are key applications of the aforementioned contravariant-induction results. To the best of our knowledge, these are the
first pushout-to-pullback theorems where pullbacks are surjective only on one side, as is the case in the mixed-pullback theorems
of \cite{cht21} and~\cite{ht23}. However, even in these mixed-pullback theorems, gluing of vertices is not allowed, ruling out
the above mentioned examples. It is worth mentioning here that one-surjective pullbacks of C*-algebras
form an ideal setting for noncommutative topology, as the Mayer--Vietoris technology still works while new types of examples
are within the scope of the theory.

As mentioned before, 
the contravariant induction  was already studied for injective graph homomorphisms. It starts in \cite[Definition~2.4.11]{aasm17}
(quotient graphs), followed by \cite{hrt20} (admissible and strongly admissible inclusions), which was recently generalized in
\cite{bs22} (breaking vertices allowed).  Our motivation comes from noncommutative topology, which includes the theory of 
\mbox{$q$-de}\-formations of 
algebras of functions on certain compact topological spaces. In~\cite{hs08}, it was shown that pushouts (unions) of graphs lead to pullback structures of the C*-algebras of even quantum
spheres. This remarkable feature was explored in~\cite{hrt20,bs22}, where 
more general pushout-to-pullback theorems were proved. Similar results can be found in the context of higher-rank graph 
C*-algebras~\cite{kpsw16} and Cuntz--Pimsner algebras~\cite{rs11}. However, when restricted to graph C*-algebras, these results
have  limited scope as they assume graphs to be without sinks (\cite{kpsw16}) or to be row finite (\cite{rs11}).

In Section~2, we consider three conditions on graph homomorphisms: properness, target bijectivity, and regularity. They turn out
to be the discrete-topology versions of Katsura's conditions \cite[Definition~2.1 and Definition~2.6]{t-k06}. We
prove that they define subcategories that in Section~4 are domains of contravariant functors for path algebras and Leavitt
path algebras, respectively. On the way, in Section~3, we systematically study pushouts of graphs proving many needed
technical results. In Section~5, we unravel new types of admissible graph homomorphism focusing on
non-injective admissible graph homomorphism.
Section~6 crowns the paper with pushout-to-pullback theorems for path algebras, Leavitt path algebras,
and graph C*-algebras. The pushout-to-pullback result for graph C*-algebras is obtained as a corollary of its Leavitt
counterpart using a beautiful theorem of Chirvasitu~\cite{a-c22}. 
Finally, we end the paper with Section~7 devoted to applications in noncommutative
topology, which involve multichamber even quantum spheres,  the Cuntz algebra~$\mathcal{O}_2$, and the
boundary quantum spheres of even Hong--Szyma\'nski quantum balls.

\section{Graphs and morphisms}

\subsection{Directed graphs}

A {\em graph} (directed graph, quiver) is a quadruple $E:=(E^0,E^1,s_E,t_E)$, where:
\begin{itemize}
\item $E^0$ is the set of {\em vertices},
\item $E^1$ is the set of {\em edges} (arrows),
\item $E^1\overset{s_E}{\to}E^0$ is the {\em source} map assigning to each edge its beginning,
\item $E^1\overset{t_E}{\to}E^0$ is the {\em target} (range) map assigning to each edge its end.
\end{itemize}

Let $v$ be a vertex in a graph $E$. It is called a {\em sink} iff $s_E^{-1}(v)=\emptyset$, it is called a {\em source} iff $t_E^{-1}(v)=\emptyset$, and it is called {\em regular} iff it is not a sink and $|s_E^{-1}(v)|<\infty$. The subset of regular vertices 
of a graph $E$ is denoted by ${\rm reg}(E)$. A {\em finite path} in $E$ is a vertex or a finite sequence $(e_1,\ldots,e_n)$ of edges 
satisfying 
\begin{equation}
t_E(e_1)=s_E(e_2),\quad t_E(e_2)=s_E(e_3),\quad \ldots,\quad t_E(e_{n-1})=s_E(e_n).
\end{equation} 
We denote the set of all finite paths in $E$ by~$FP(E)$. The beginning $s_E(p_n)$ of $p_n$ is $s_E(e_1)$ and the end $t_E(p_n)$ of 
$p_n$ is $t_E(e_n)$. The beginning and the end of a vertex is the vertex itself. Thus we extend the source and target maps to 
$s_{PE},t_{PE}\colon FP(E)\to E^0$. Vertices are considered as finite paths of length~$0$. The {\em length} of a finite path that is 
not a vertex is the number of elements in the sequence. In particular, every edge is a path of length~$1$. We denote the set of all paths of length $n$ 
by~$FP_n(E)$.

%%%%%%%%%%%%%%%%%%%%%%%%%%%%%%%%%%%%%%%%%%%%%%%%%%%%%%%%%%%%%%

\subsection{Categories of graphs}

Let $E:=(E^0,E^1,s_E,t_E)$ and $F:=(F^0,F^1,s_F,t_F)$ be graphs. A~{\em homomorphism} from $E$ to $F$ is a pair of maps 
\begin{equation}
(f^0:E^0\to F^0,f^1:E^1\to F^1)
\end{equation} 
satisfying the conditions:
\begin{equation} \label{graphom}
s_F\circ f^1=f^0\circ s_E\,,\qquad t_F\circ f^1=f^0\circ t_E\,.
\end{equation}
We denote the category of graphs and graph homomorphisms by {\rm OG} and call it the {\em standard category of graphs}.
We call a graph homomorphism $(f^0,f^1)$ \emph{injective} or \emph{surjective} iff both $f^0$ and $f^1$ are injective or 
surjective, respectively.

If $(f^0,f^1)\colon E\to F$ is a homomorphism of graphs, then we define $f\colon FP(E)\to FP(F)$ as follows
\begin{gather}
\forall\; v\in E^0\colon f(v):=f^0(v),\quad
\forall\; e\in E^1\colon f(e):=f^1(e),\nonumber\\
\forall\; (e_1,\ldots,e_n)\in FP(E)\colon f( (e_1,\ldots,e_n)):= (f^1(e_1),\ldots,f^1(e_n))\in FP(F).
\end{gather}
If $(f^0,f^1)$ is injective or surjective, then so is $f$. Note also that now we can think of $FP$ as a covariant
functor from the category $\mathrm{OG}$ of graphs and graph homomorphisms to the category of sets and maps.

\begin{definition}
A {\em proper homomorphism} of graphs $f\colon E\to F$ is a homomorphism of graphs
whose both maps are finite-to-one, i.e.\
\begin{equation*}
\forall\;v\in F^0\colon |(f^0)^{-1}(v)|<\infty,\quad
\forall\;e\in F^1\colon |(f^1)^{-1}(e)|<\infty.
\end{equation*}
We denote the category of graphs and proper graph homomorphisms by {\rm POG}. 
\end{definition}
\noindent
First, observe that {\rm POG} is indeed a subcategory of {\rm OG} due to the fact that the
 composition of finite-to-one maps is again a 
finite-to-one map. Moreover, if $(f^0,f^1)\colon E\to F$ is a proper homomorphism of
 graphs, then the induced map $f\colon FP(E)\to FP(F)$ is finite to one. Indeed, if $p\in FP(E)$ is a vertex, 
then $f^{-1}(p)=(f^0)^{-1}(p)$ is a finite set. 
Next, let $p=(p_1,\ldots,p_n)$, $p_i\in F^1$ for all $1\leq i\leq n$,  and $q\in f^{-1}(p)$. Then 
we can write $q=(q_1,\ldots,q_n)$, $q_i\in E^1$ for all $1\leq i\leq n$, and
$f(q)=(f^1(q_1),\ldots,f^1(q_n))=(p_1,\ldots,p_n)$. Hence, $q_i\in (f^1)^{-1}(p_i)$ for all $1\leq i\leq n$,
so the number of elements in $f^{-1}(p)$ is limited by the number of elements in 
\begin{equation}\label{finitetoone}
(f^1)^{-1}(p_1)\times\ldots\times (f^1)^{-1}(p_i)\times\ldots\times (f^1)^{-1}(p_n),
\end{equation}
which is a finite set. Finally, observe also that, much as before, 
we can view $FP$ as a covariant functor from the category {\rm POG} to the category of sets and finite-to-one maps.

\begin{definition}
We say that a graph homomorphism $(f^0,f^1):E\to F$ satisfies the {\em target-injectivity (target-surjectivity) condition} if
\begin{equation}\label{EU}
\forall\; x\in F^1\colon (f^1)^{-1}(x)\ni e\longmapsto t_E(e) \in (f^0)^{-1}(t_F(x))\;\text{ is injective (surjective)}.
\end{equation}
We say that $(f^0,f^1)$ satisfies the {\em target-bijectivity condition} if it satisfies both the target-injectivity condition
and the target-surjectivity condition.
\end{definition}
\noindent
Note that the bijectivity of $(f^0,f^1)$ implies the target bijectivity of~$(f^0,f^1)$, so \eqref{EU} is satisfied for 
$(\id_{E^0},\id_{E^1})$. However, 
an injective homomorphism of graphs need not satisfy the target-bijectivity condition. Indeed, mapping
the one-vertex graph into the one-loop graph by assigning the vertex to the base of the loop is an
injective graph homomorphism but the target-bijectivity condition fails. 
Next, a graph homomorphism $(f^0,f^1)\colon E\to F$ that is 
injective on vertices and satisfies the target-bijectivity 
condition is injective: if $e_1$ and $e_2$ are edges such that $f^1(e_1)=f^1(e_2)$, then $e_1,e_2\in (f^1)^{-1}(f^1(e_1))$
and $|(f^0)^{-1}\big(t_F(f^1(e_1))\big)|=1$, so $e_1=e_2$.

Next, we present a more conceptual version of the target-bijectivity condition.
\begin{proposition}\label{targetpull}
A graph homomorphism $(f^0,f^1):E\to F$ satisfies the target-bijectivity condition if and only if the commutative diagram
\begin{equation}\label{pulltarget}
\xymatrix{
&
E^1
\ar[dl]_{t_E}
\ar[dr]^{f^1}
&\\
E^0 \ar[dr]_{f^0}
& & 
F^1 \ar[dl]^{t_F}
\\
&
F^0
&
}
\end{equation}
given by~\eqref{graphom} is a pullback diagram in the category of sets and maps.
\end{proposition}
\begin{proof}
Recall that the pullback of $f^0:E^0\to F^0$ and $t_F:F^1\to F^0$ in the category of sets and maps is the fibered product
\begin{equation}
E^0\underset{F^0}{\times}F^1:=\{(v,x)\in E^0\times F^1~|~f^0(v)=t_F(x)\}
\end{equation}
together with the projections onto each component. Since the diagram \eqref{pulltarget} is commutative, the universal
property of the pullback manifests itself in the existence of the map
\begin{equation}
\Phi:E^1\longrightarrow E^0\underset{F^0}{\times}F^1,\qquad
e\longmapsto (t_E(e),f^1(e)).
\end{equation}
We have to prove that $\Phi$ is a bijection $\iff$ $(f^0,f^1)$ satisfies the target-bijectivity condition. In fact, we will prove that 
$\Phi$ is injective if and only if $(f^0,f^1)$ satisfies the target-injectivity condition, and that 
$\Phi$ is surjective if and only if $(f^0,f^1)$ satisfies the target-surjectivity condition.

First, note that \eqref{EU} defines a family of maps labelled by $x\in F^1$:
\begin{equation}
E^1\supseteq (f^1)^{-1}(x)\ni e\stackrel{\Phi_x}{\longmapsto}
 (t_E(e),f^1(e)) \in \big((f^0)^{-1}(t_F(x)),x\big)\subseteq E^0\underset{F^0}{\times}F^1.
\end{equation}
It is clear that the target-injectivity of $(f^0,f^1)$ is equivalent to the injectivity of $\Phi_x$ for all $x\in F^1$,
and that the target-surjectivity of $(f^0,f^1)$ is tantamount to the surjectivity of $\Phi_x$ for all $x\in F^1$.
Observe also that, if $x\neq y$, then $\Phi_x$ and $\Phi_y$ have disjoint domains and counterdomains, and the union of all domains 
is~$E^1$. Now, since $\Phi$ agrees with $\Phi_x$ on the domain of the latter for any $x\in F^1$, 
it follows immediately that the target-injectivity condition of $(f^0,f^1)$
is equivalent to the injectivity of~$\Phi$.

Assume next the target-surjectivity of~$(f^0,f^1)$. This implies that
the union of the counterdomains of $\Phi_x$ is $E^0{\times}_{F^0}F^1$. Indeed, take any
$(v,x)\in E^0{\times}_{F^0}F^1$. Then $v\in (f^0)^{-1}(t_F(x))$, and there exists $e\in E^1$ such that $t_E(e)=v$
and $f^1(e)=x$, so $(v,x)\in \Phi_x((f^1)^{-1}(x))$.
Now one can see that 
\begin{equation}
\Phi=\bigsqcup_{x\in F^1}\Phi_x,
\end{equation}
so the target-surjectivity of~$(f^0,f^1)$ implies the surjectivity of~$\Phi$. Vice versa, since the image of $\Phi$ is
contains  the union of the counterdomains of $\Phi_x$, it is immediate that the surjectivity of~$\Phi$ implies
the target-surjectivity of~$(f^0,f^1)$.
\end{proof}

We can now easily claim the desired composability of the target-bijectivity condition:
\begin{lemma}\label{tbpog}
Restricting morphisms of the category {\rm POG} to the morphisms satisfying
 the  target-bijectivity condition yields a subcategory of~{\rm POG}.
\end{lemma}
\begin{proof}
Let $(f^0,f^1):E\to F$ and $(g^0,g^1):F\to G$ be morphisms in {\rm POG}. 
We already know that $(g^1\circ f^1,g^0\circ f^0)\in {\rm Mor}(\mathrm{POG})$. Furthermore, we have the following 
commutative diagram:
\begin{equation}
\xymatrix{
E^1 \ar[r]^{f^1}
\ar[d]_{t_E}
& F^1 \ar[d]_{t_F} \ar[r]^{g^1} & G^1 \ar[d]^{t_G}\\
E^0 \ar[r]_{f^0}
& F^0 \ar[r]_{g^0} & \phantom{.}G^0.
}
\end{equation}
Now it follows from standard category theory (e.g., see~\cite[Proposition~11.10]{ahs90}) that, 
if both squares are pullback diagrams, then the outer rectangle is also a pullback diagram, which ends the proof
by Proposition~\ref{targetpull}.
\end{proof}

We denote the subcategory of {\rm POG}  from Lemma~\ref{tbpog} by {\rm TBPOG}. This subcategory enjoys the following useful result.

\begin{lemma}\label{targetpath}
Let $f:E\to F$ be a morphism in TBPOG and $\alpha\in FP(E)$. The map
\[
f^{-1}(f(\alpha))\ni\beta\longmapsto t_{PE}(\beta)\in f^{-1}(f(t_{PE}(\alpha)))
\]
is bijective for all $\alpha\in FP(E)$.
\end{lemma}
\begin{proof} 
Note first that the statement is clearly true when $\alpha$ is a vertex. Assume now that $\alpha:=(e_1,\ldots, e_n)$, where all $e_i$ are edges. 
Then
\begin{equation}
\forall\,v\in (f^0)^{-1}(f^0(t_E(e_n)))~\exists!\,x_n\in (f^1)^{-1}(f^1(e_n)):\quad t_E(x_n)=v
\end{equation}
by the target-bijectivity condition. Much in the same way,
\begin{equation}
\forall\,x_i\in (f^1)^{-1}(e_i)~\exists!\,x_{i-1}\in (f^1)^{-1}(e_{i-1}):\quad t_E(x_{i-1})=s_E(x_i),\qquad i\in\{2,\ldots,n\}.
\end{equation}
Thus, every vertex $v\in (f^0)^{-1}(f^0(t_E(e_i)))$ generates a unique path in $f^{-1}(f(\alpha))$ ending at $v$, i.e. the map
\begin{equation}
f^{-1}(f(\alpha))\ni\beta\longmapsto t_{PE}(\beta)\in f^{-1}(f(t_{PE}(\alpha)))
\end{equation}
is bijective.
\end{proof}

\begin{definition}
A {\em regular homomorphism} of graphs $(f^0,f^1)\colon E\to F$ is a homomorphism of graphs satisfying the  condition
\begin{equation}\label{CR}
f^0\big(E^0\setminus\mathrm{reg}(E)\big)\subseteq F^0\setminus\mathrm{reg}(F).
\end{equation}
\end{definition}
\noindent
Note that \eqref{CR} can be equivalently written as
\begin{equation}\label{214}
(f^0)^{-1}(\mathrm{reg}(F))\subseteq\mathrm{reg}(E).
\end{equation}
Moreover, it is clear that the identity is a regular homomorphism and that a composition of regular homomorphisms is regular. Thus, there 
exists a subcategory of {\rm TBPOG} given by restricting morphisms therein to regular graph homomorphisms.
We denote the category of graphs and regular proper  homomorphisms of graphs satisfying the target-bijectivity condition by 
{\rm CRTBPOG}, 
and call it the  {\em admissible category of graphs}. Morphisms in this category are called \emph{admissible graph homomorphisms}.

\begin{example}
{\rm Let 
$$
E:=\begin{tikzpicture}[auto,swap]
\tikzstyle{vertex}=[circle,fill=black,minimum size=3pt,inner sep=0pt]
\tikzstyle{edge}=[draw,->]
   
\node[vertex,label=below:$v$] (0) at (1,0) {};
\node[vertex,label=above:$w_1$] (1) at (0,0.5) {};
\node[vertex,label=below:$w_2$] (2) at (0,-0.5) {}; 

\path (1) edge[edge] node[above] {$e_1$} (0);
\path (2) edge[edge] node[below] {$e_2$} (0);

\end{tikzpicture}\quad \text{and} \quad
F:=\begin{tikzpicture}[auto,swap]
\tikzstyle{vertex}=[circle,fill=black,minimum size=3pt,inner sep=0pt]
\tikzstyle{edge}=[draw,->]
   
\node[vertex,label=below:$a$] (0) at (0,0) {};
\node[vertex,label=below:$c$] (1) at (1,0) {}; 

\path (0) edge[edge] node[above] {$x$} (1);

\end{tikzpicture}.
$$
Then mapping the vertices $w_1$ and $w_2$ to $a$, the edges $e_1$ and $e_2$ to $x$, and the vertex $v$ to $c$, 
defines a regular proper  graph homomorphism that 
does not satisfy the target-bijectivity condition.}
\end{example}

\begin{example}\label{stallings}
{\rm Let 
$$
E:=\begin{tikzpicture}[auto,swap]
\tikzstyle{vertex}=[circle,fill=black,minimum size=3pt,inner sep=0pt]
\tikzstyle{edge}=[draw,->]
   
\node[vertex,label=below:$v$] (0) at (0,0) {};
\node[vertex,label=above:$w_1$] (1) at (1,0.5) {};
\node[vertex,label=below:$w_2$] (2) at (1,-0.5) {}; 

\path (0) edge[edge] node[above] {$e_1$} (1);
\path (0) edge[edge] node[below] {$e_2$} (2);

\end{tikzpicture}\quad \text{and} \quad
F:=\begin{tikzpicture}[auto,swap]
\tikzstyle{vertex}=[circle,fill=black,minimum size=3pt,inner sep=0pt]
\tikzstyle{edge}=[draw,->]
   
\node[vertex,label=below:$a$] (0) at (0,0) {};
\node[vertex,label=below:$c$] (1) at (1,0) {}; 

\path (0) edge[edge] node[above] {$x$} (1);

\end{tikzpicture}.
$$
Then mapping the vertex $v$ to $a$, the vertices $w_1$ and $w_2$ to $c$, and the edges $e_1$ and $e_2$ to $x$,
 defines an admissible graph homomorphism. This is an elementary example of an unlabeled Stallings folding~\cite{s-jr83}.
}
\end{example}

%%%%%%%%%%%%%%%%%%%%%%%%%%%%%%%%%%%%%%%%%%%%%%%%%%%%%%%%%%%%%%

\section{Pushouts of graphs}
\noindent
We refer the reader to~\cite{ek79} for an extensive study of pushouts of directed graphs.
\subsection{Unions of graphs}
We begin with unions of graphs, which are the simplest examples of pushouts of graphs.
Let $E$ and $F$ be directed graphs. If there is an injective graph homomorphism $(f^0,f^1):E\hookrightarrow F$ given by inclusions,
 then we say that 
$E$ is a {\em subgraph} of $F$, which we write $E\subseteq F$. 
Next, let $F$ and $G$ be graphs.
 Assume that $s_F$ and $t_F$ agree, respectively, with $s_G$ and $t_G$ on $F^1\cap G^1$.
Then we can define the {\em intersection} graph 
\begin{equation}
F\cap G:=(F^0\cap G^0,F^1\cap G^1,s_\cap,t_\cap),
\end{equation}
where $s_\cap,t_\cap:F^1\cap G^1\to F^0\cap G^0$ are given by
\begin{equation}
\forall\; e\in F^1\cap G^1:\quad s_\cap(e):=s_G(e)=s_F(e),\quad t_\cap(e):=t_G(e)=t_F(e).
\end{equation}
Next, we can define the {\em union} graph 
\begin{equation}
F\cup G:=(F^0\cup G^0,F^1\cup G^1,s_\cup,t_\cup),
\end{equation}
where $s_\cup,t_\cup:F^1\cup G^1\to F^0\cup G^0$ are given by
\begin{equation*}
\forall\; e\in F^1\cup G^1:\quad 
s_\cup(e):=
\begin{cases}
s_F(e)&\text{for}\;e\in F^1,\\
s_G(e)&\text{for}\;e\in G^1,
\end{cases}
\quad\text{and}\quad
t_\cup(e):=
\begin{cases}
t_F(e)&\text{for}\;e\in F^1,\\
t_G(e)&\text{for}\;e\in G^1.
\end{cases}
\end{equation*}
The intersection graph $F\cap G$ is a subgraph of both $F$ and $G$, and both $F$ and $G$ are subgraphs of the union graph 
$F\cup G$. The intersection graph $F\cap G$ exists if and only if the union graph $F\cup G$ exists.

Now, we recall the concept of hereditary and saturated subsets of the set of vertices in a graph. Let $E$ be a graph. A subset 
$H\subseteq E^0$ is called {\em hereditary}
if any edge starting at $v\in H$ ends at $w\in H$, and it is called {\em saturated} if there does not exist a 
regular vertex $v\in E^0\setminus H$ such that 
$t_E(s_E^{-1}(v))\subseteq H$. 
Note that in the above definition of a hereditary subset one can replace the word 
``edge'' by the phrase ``finite path''. Observe also that the formulas $s_F(e):=s_E(e)$, $t_F(e):=t_E(e)$, 
$e\in F^1:=E^1\setminus t_E^{-1}(H)$,
define a~subgraph $F$ of $E$ with $F^0:=E^0\setminus H$ if and only if $H$ is hereditary.

Furthermore, we say that $v\in E^0$ is a {\em breaking vertex} for $H$ iff
\begin{equation}
v\in E^0\setminus H,\qquad |s^{-1}_E(v)|=\infty,\qquad\text{and}\qquad 0<|s^{-1}_E(v)\cap t^{-1}_E(E^0\setminus H)|<\infty.
\end{equation}
We denote the set of all breaking vertices for $H$ by
\begin{equation}
B_H:=\{v\in E^0\setminus H~|~v\text{ is a breaking vertex for }H\}.
\end{equation}
A subset $H$ of $E^0$ is called \emph{unbroken} if and only if $B_H=\emptyset$. Note that a breaking vertex of $H$ becomes 
regular in the subgraph obtained by removing all vertices
in $H$ and all edges ending in~$H$.

We are ready now to bundle up the three properties of being hereditary, saturated and unbroken
to restrict subgraphs to these that played a crucial role in~\cite{hrt20}.
\begin{definition}\label{admissiblesubgraph}
An injective graph homomorphism $(f^0,f^1):E\hookrightarrow F$ is called $\cup$-{\em admissible} iff
 it satisfies the following conditions:
\begin{enumerate}
\item[{\rm (A1)}] $F^0\setminus f^0(E^0)$ is saturated,
\item[{\rm (A2)}] $t^{-1}_F(f^0(E^0))\subseteq f^1(E^1)$.
\end{enumerate}
We call a $\cup$-admissible injective graph homomorphism
 \emph{strongly $\cup$-admissible} iff, in addition, the subset $F^0\setminus f^0(E^0)$ is unbroken.
In the case the maps defining a (strongly) $\cup$-admissible injective graph homomorphism $(f^0,f^1)$ are inclusions, 
we call $E$ a {\em (strongly) admissible subgraph} of~$F$.
Furthermore, we call intersecting  graphs $F$ and $G$ (strongly) admissible if both inclusions $F\cap G\subseteq F$ 
and $F\cap G\subseteq G$ are (strongly) $\cup$-admissible. Much in the same way, we call taking the union of graphs $F$ and $G$ 
(strongly) admissible if both inclusions $F\subseteq F\cup G$ 
and $G\subseteq F\cup G$ are (strongly) $\cup$-admissible. 
\end{definition}
\noindent
The above definition already appeared in~\cite[Definition~3.1]{cht21} (see also~\cite[Definition~2.1]{hrt20}), where it is also 
assumed that $f^1(E^1)\subseteq t^{-1}_F(f^0(E^0))$ and $F^0\setminus f^0(E^0)$ is hereditary. However, the 
first condition is always true for any graph homomorphism and the hereditarity follows from 
the condition~(A2):
\begin{proposition}\label{hermorph}
Let $(f^0,f^1):E\to F$ be any graph homomorphism satisfying~(A2). Then $F^0\setminus f^0(E^0)$ is hereditary.
\end{proposition}
\begin{proof}
Suppose that $F^0\setminus f^0(E^0)$ is not hereditary, i.e.\ there is $x\in F^1$ such that $s_F(x)\in F^0\setminus f^0(E^0)$ and 
$t_F(x)\in f^0(E^0)$. Since $t_F^{-1}(f^0(E^0))\subseteq f^1(E^1)$, there is an edge $e\in E^1$ such that $f^1(e)=x$. 
Then $s_F(x)=s_F(f^1(e))=f^0(s_E(e))$, which gives a contradiction.
\end{proof}

Next, we turn to unbroken subsets. Our next result shows that the assumption of strong admissibility of taking the union 
in~\cite[Theorem~3.1]{hrt20} is superfluous.
\begin{lemma}\label{captocup}
Let $F$ and $G$ be arbitrary graphs whose source and target maps agree, respectively, on $F^1\cap G^1$. Then, if intersecting  $F$ 
and $G$ is strongly admissible, so is taking the union of $F$ and~$G$.
\end{lemma}
\begin{proof}
Assume that  intersecting $F$ and $G$ is strongly admissible. Then, to prove that also taking the union 
of $F$ and $G$
is strongly admissible, it suffices to show that both $(F^0\cup G^0)\setminus F^0$ and
\mbox{$(F^0\cup G^0)\setminus G^0$}
are unbroken in $F\cup G$. To this end, suppose that 
$v\in F^0$ is a breaking vertex for $(F^0\cup G^0)\setminus F^0$  in $F\cup G$. Then
$v$ emits infinitely many, whence not zero, edges ending in $G^0\setminus F^0$, so $v\in F^0\cap G^0$. Also,
all these edges are from~$G^1$.
Furthermore, $v$ emits at least one and at most finitely many edges ending in~$F^0$. If all of them 
end in $F^0\setminus G^0$, then they are all from~$F^1$, and they render $F^0\setminus G^0$
not saturated in~$F$, which is not allowed by the $\cup$-admissibility of $(F\cap G)\subseteq F$. Therefore,
$v$ emits at least one edge $e$ ending in $F^0\cap G^0$. 
Now, from the $\cup$-admissibility of $(F\cap G)\subseteq F$ and 
$(F\cap G)\subseteq G$,
we obtain
\begin{equation}
e\in t_\cup^{-1}(F^0\cap G^0)=t_F^{-1}(F^0\cap G^0)\cup t_G^{-1}(F^0\cap G^0)=F^1\cap G^1,
\end{equation}
 so $e\in G^1$. Also, since $v$ emits only finitely
many edges into $F^0$, in particular it emits only finitely many edges from $G^1$ ending in $F^0\cap G^0$.
All this makes $v$ a breaking vertex for $G^0\setminus F^0$ in~$G$, which contradicts
 the strong $\cup$-admissibility of $(F\cap G)\subseteq G$. Hence, $(F^0\cup G^0)\setminus F^0$ 
is unbroken in $F\cup G$. Finally, the symmetric argument shows that  $(F^0\cup G^0)\setminus G^0$ 
is unbroken in~$F\cup G$.
\end{proof}
 
Now, let us show that $\cup$-admissible injective graph homomorphisms are special cases of morphisms in the category CRTBPOG. 
This is why we call CRTBPOG 
the admissible category of graphs.
\begin{proposition}
Let $(f^0,f^1):E\hookrightarrow F$ be an injective graph homomorphism. 
Then $(f^0,f^1)$ is $\cup$-admissible if and only if it is admissible.
\end{proposition}
\begin{proof}
($\Rightarrow$) Since $(f^0,f^1)$ is injective, it is clearly proper. Next, we check the regularity of~$(f^0,f^1)$.
Since infinite emitters in the subgraph remain infinite emitters in the graph, it suffices to consider the images of sinks. 
Assume that $v\in E^0$ is a sink 
and suppose that $f^0(v)$ is 
regular in $F$. First, note that $t_F(s^{-1}_F(f^0(v)))\subseteq F^0\setminus f^0(E^0)$.
Indeed, take an edge $y\in F^1$ such that $s_F(y)=f^0(v)$ and suppose that $t_F(y)\in f^0(E^0)$. Then, by (A2), there is an edge 
$e\in E^1$ such that $f^1(e)=y$. However, since $f^0(s_E(e))=s_F(f^1(e))=f^0(v)$, by the injectivity of~$f^0$, we obtain that 
$s_E(e)=v$, which contradicts the assumption that $v$ is a sink. Consequently, what we have just proved 
contradicts the fact that $F^0\setminus f^0(E^0)$ is saturated (the condition (A1)). Finally, we 
have to show that the target-bijectivity condition is satisfied.
Due to the injectivity of $(f^0,f^1)$, we know that, for any $x\in F^1$, the sets $(f^1)^{-1}(x)$ and $(f^0)^{-1}(t_F(x))$ are either 
empty or consist of a single element. To prove the claim, it suffices to exclude the possibility in which one of these sets is empty and 
the other one is not. First, if $(f^1)^{-1}(x)=\{e\}$, then $t_E(e)\in(f^0)^{-1}(t_F(x))$. Next, since 
$t_F^{-1}(f^0(E^0))\subseteq f^1(E^1)$, if $(f^0)^{-1}(t_F(x))=\{v\}$, then $(f^1)^{-1}(x)\neq\emptyset$.

($\Leftarrow$) Assume that $(f^0,f^1):E\hookrightarrow F$ is a morphism in CRTBPOG. First, let us prove that 
$t_F^{-1}(f^0(E^0))\subseteq f^1(E^1)$. Let $x\in F^1$ be such that $t_F(x)\in f^0(E^0)$. This implies that there is $v\in E^0$ 
such that $f^0(v)=t_F(x)$. In turn, from the target-bijectivity condition, we infer that there exists $e\in (f^1)^{-1}(x)$ such that 
$t_E(e)=v$. Now it suffices to prove that \mbox{$F^0\setminus f^0(E^0)$} is saturated. 
To this end, suppose that there is a~regular vertex 
$w\in f^0(E^0)$ such that $t_F(s^{-1}_F(w))\subseteq F^0\setminus f^0(E^0)$. It follows that 
$s^{-1}_F(w)\cap f^1(E^1)=\emptyset$. Indeed, suppose that there is an edge $y\in F^1$ such that $s_F(y)=w$ and $f^1(e)=y$
 for some $e\in E^1$. Then $t_F(y)=t_F(f^1(e))=f^0(t_E(e))$, which contradicts the assumption. Therefore, if $v\in E^0$ is 
a vertex such that $f^0(v)=w$, then  $s^{-1}_E(v)=\emptyset$, which contradicts the regularity of $(f^0,f^1)$.
\end{proof}

\begin{remark} \label{remark3.5}
{\rm Note that, if $f:=(f^0,f^1):E\to F$ is an admissible graph homomorphism, then $f(E)$ is an admissible subgraph of $F$. 
Indeed, suppose that $f^0(v)\in {\rm reg}(F)$ and 
\begin{equation}
t_F(s_F^{-1}(f^0(v)))\subseteq F^0\setminus f^0(E^0).
\end{equation}
Then $v\in {\rm reg}(E)$, so there exists 
$e\in s_E^{-1}(v)$. Furthermore, $f(e)\in s_F^{-1}(f^0(v))$ but 
\begin{equation}
t_F(f^1(e))=f^0(t_E(e))\in f^0(E^0),
\end{equation}
which gives a contradiction, and proves 
Definition~3.1(A1). To prove Definition~3.1(A2), take any $x\in t_F^{-1}(f^0(E^0))$. (If $t_F^{-1}(f^0(E^0))=\emptyset$, we are done.) 
Then $(f^0)^{-1}(t_F(x))\neq\emptyset$ so by the target-surjectivity of $f$, there exists $e\in (f^1)^{-1}(x)$. Hence, $x\in f^1(E^1)$, as needed. 
However, restricting the codomain of $f$ to its image might yield a  graph homomorphism $E\to f(E)$ that is not  admissible:
\begin{equation}
\begin{tikzpicture}[auto,swap]
\tikzstyle{vertex}=[circle,fill=black,minimum size=3pt,inner sep=0pt]
\tikzstyle{edge}=[draw,->]
\tikzstyle{cycle1}=[draw,->,out=130, in=50, loop, distance=30pt]
   
\node[vertex] (0) at (0,0) {};
\node[vertex] (1) at (0,-1) {};
\node[vertex] (2) at (1,-1) {};

\path (1) edge[draw,->] node[above] {$e$} (2);
\end{tikzpicture}\quad 
\begin{tikzpicture} 
\node (0) at (0,1) {};
\node (1) at (1,1) {};
\path (0) edge[draw,->] node[above] {} (1);
\end{tikzpicture}\quad
\begin{tikzpicture}[auto,swap]
\tikzstyle{vertex}=[circle,fill=black,minimum size=3pt,inner sep=0pt]
\tikzstyle{edge}=[draw,->]
\tikzstyle{cycle1}=[draw,->,out=130, in=50, loop, distance=30pt]

\node (-1) at (0,0) {};   
\node[vertex] (0) at (0,-1) {};
\node[vertex] (1) at (1,-1) {};

\path (0) edge[edge] node[above] {$e$} (1);

\end{tikzpicture}\quad
\begin{tikzpicture} 
\node (0) at (0,1) {};
\node (1) at (1,1) {};
\path (0) edge[draw,->] node[below] {${}$} (1);
\end{tikzpicture}\quad
\begin{tikzpicture}[auto,swap]
\tikzstyle{vertex}=[circle,fill=black,minimum size=3pt,inner sep=0pt]
\tikzstyle{edge}=[draw,->]
\tikzstyle{cycle1}=[draw,->,out=130, in=50, loop, distance=30pt]
   
\node[vertex] (0) at (0,0) {};
\node[vertex] (1) at (1,0) {};
\node[vertex] (2) at (0,1) {};

\path (0) edge[edge] node[above] {$e$} (1);

\path (0) edge[edge] node[left] {${\tiny \infty}$} (2);
\end{tikzpicture}\quad.
\end{equation} 
\[
\quad E\qquad \longrightarrow\qquad f(E)\qquad\longrightarrow\qquad F\qquad 
\]}
\end{remark}

%%%%%%%%%%%%%%%%%%%%%%%%%%%%%%%%%%%%%%%%%%%%%%%%%%%%%%%%%%%%%%

\subsection{Pushouts of graphs in different categories}
In the category of sets and maps, the pushout of $X\stackrel{f}{\leftarrow} Z\stackrel{g}{\to} Y$ is
\begin{equation}
X\stackrel{\iota_X}{\longrightarrow}X\underset{Z}{\amalg} Y\stackrel{\iota_Y}{\longleftarrow}Y,
\quad X\underset{Z}{\amalg} Y:=(X\amalg Y)/R_Z\,,
\end{equation}
where $R_Z$ is the minimal equivalence relation generated by $f(z)R_Zg(z)$, $z\in Z$, and
$\iota_X$ and $\iota_Y$ are the obvious induced maps. We call a pushout diagram {\em one-injective} whenever at least one of
 the defining maps is injective.

The above pushout construction does not always yield a pushout in the category of sets and finite-to-one maps. Therefore, we 
we need the following elementary result:
\begin{lemma}\label{oneinjective}
Let 
\begin{equation*}
\xymatrix{
&
P
&\\
X
\ar[ur]^{\iota_X}& & 
Y
\ar[ul]_{\iota_Y}\\
&
Z
\ar[ur]_{g}\ar[ul]^{f}&
}
\end{equation*}
be a pushout diagram in the category of sets and maps. If one of the maps $f$ and $g$ is injective and
the other one is finite to one, then the above diagram is a pushout diagram in the category of sets and 
finite-to-one maps.
\end{lemma}
\begin{proof}
Assume without the loss of generality that $f$ is injective and $g$ is finite to one. Then the canonical maps 
of their pushout 
\begin{equation}
X\stackrel{\iota_X}{\longrightarrow}X\underset{Z}{\amalg} Y\stackrel{\iota_Y}{\longleftarrow}Y
\end{equation}
are also finite to one. Indeed, if $\iota_Y(y)=\iota_Y(y')$, then $y=y'$ or there exists a finite
sequence $(z_1,\ldots, z_{2n})\in Z^{2n}$ such that
\begin{align}
y=g(z_1)&\text{ and }f(z_1)=x_1=f(z_2),
\nonumber\\
g(z_2)=y_1=g(z_3)&\text{ and }f(z_3)=x_2=f(z_4),
\nonumber\\
&\;\;\;\vdots
\nonumber\\ 
g(z_{2n-2})=y_{n-1}=g(z_{2n-1})&\text{ and }f(z_{2n-1})=x_{n}=f(z_{2n}),
\nonumber\\ 
&g(z_{2n})=y'.
\end{align}
In the latter case, from the injectivity of $f$ we  conclude that $z_{2k-1}=z_{2k}$ for all $k\in\{1,\ldots,n\}$, so 
\begin{equation}
y=g(z_1)=g(z_{2})=y_1=\dots=y_{n-1}=g(z_{2n-1})=g(z_{2n})=y'.
\end{equation}
 Hence,  $\iota_Y$ is injective.
Next, if $\iota_X(x)=\iota_X(x')$, then $x=x'$ or there exists a finite
sequence $(z_1,\ldots, z_{2m})\in Z^{2m}$ such that
\begin{align}
x=f(z_1)&\text{ and }g(z_1)=y_1=g(z_2),
\nonumber\\
f(z_2)=x_1=f(z_3)&\text{ and }g(z_3)=y_2=g(z_4),
\nonumber\\
&\;\;\;\vdots
\nonumber\\ 
f(z_{2m-2})=x_{m-1}=f(z_{2m-1})&\text{ and }g(z_{2m-1})=y_{m}=g(z_{2m}),
\nonumber\\ 
&f(z_{2m})=x'.
\end{align}
It follows from the injectivity of $f$ that $z_{2k}=z_{2k+1}$ for all $k\in\{1,\ldots,m-1\}$, 
so 
\begin{equation}
g(z_1)=y_1=g(z_2)=g(z_3)=y_2=\ldots=g(z_{2m-2})=g(z_{2m-1})=y_{m}.
\end{equation}
Therefore,
all $y_i$ are equal to $g(z_1)$, where $z_1$ is uniquely determined
by~$x$. Hence, also $y_m$ is uniquely determined by $x$. Consequently, as $x'\in f(g^{-1}(y_m))$, which is a finite set by assumption, there are only finitely many $x'$ such that $\iota_X(x)=\iota_X(x')$, so the map $\iota_X$ is finite to one.

Finally, if $j_X\colon X\to Q$ and $j_Y\colon Y\to Q$ are finite-to-one maps such 
that $j_X\circ f=j_Y\circ g$, then the universal-property map $h\colon P\to Q$  is also finite to one.
Indeed, suppose that the set $h^{-1}(q)$ is infinite for some $q\in Q$. Then, as $P=\iota_X(X)\cup\iota_Y(Y)$,
one of the sets $h^{-1}(q)\cap\iota_X(X)$ and $h^{-1}(q)\cap\iota_Y(Y)$ is infinite, 
which contradicts the assumption
that both $j_X=h\circ\iota_X$ and $j_Y=h\circ\iota_Y$ are finite to one.
\end{proof}

Let $E\stackrel{(f^0,f^1)}{\longleftarrow} G\stackrel{(g^0,g^1)}{\longrightarrow} F$ 
be graph homomorphisms and let $E^i\amalg_{G^i}F^i$ be the coresponding pushout of 
\mbox{$E^i\stackrel{f^i}{\leftarrow} G^i\stackrel{g^i}{\to} F^i$}, $i=0,1$, in the category of sets. Then
we have the following  commutative diagrams:
\begin{equation}
\xymatrix{
&G^1\ar[ld]_{f^1}\ar[rd]^{g^1}\ar@/^0.5pc/[rrrr]^{ s_G } && & & G^0\ar[ld]_{f^0}\ar[rd]^{g^0} &\\
E^1\ar[rd]_{\iota_{E^1}}\ar@/_1pc/[rrrr]^{s_E}
 && F^1\ar[ld]^{\iota_{F^1}}\ar@/^1pc/[rrrr]_{s_F}
 && E^0 \ar[rd]_{\iota_{E^0}}&& F^0\,,\ar[ld]^{\iota_{F^0}}\\
&E^1\underset{G^1}{\amalg}F^1\ar@/_0.5pc/[rrrr]_{s_\amalg}& && & E^0\underset{G^0}{\amalg}F^0&
}
\end{equation}
\begin{equation}
\xymatrix{
&G^1\ar[ld]_{f^1}\ar[rd]^{g^1}\ar@/^0.5pc/[rrrr]^{ t_G } && & & G^0\ar[ld]_{f^0}\ar[rd]^{g^0} &\\
E^1\ar[rd]_{\iota_{E^1}}\ar@/_1pc/[rrrr]^{t_E}
 && F^1\ar[ld]^{\iota_{F^1}}\ar@/^1pc/[rrrr]_{t_F}
 && E^0 \ar[rd]_{\iota_{E^0}}&& F^0\,.\ar[ld]^{\iota_{F^0}}\\
&E^1\underset{G^1}{\amalg}F^1\ar@/_0.5pc/[rrrr]_{t_\amalg}& && & E^0\underset{G^0}{\amalg}F^0&
}
\end{equation}
Here the left and right square subdiagrams commute by the definition of a pushout, and the top
subdiagrams commute by the definition of a graph homomorphism. Moreover,
$s_\amalg$ and $t_\amalg$ exist by the universal property of the pushout 
$E^1\amalg_{G^1}F^1$, which applies
due to the equalities
\begin{align}
&\iota_{E^0}\circ s_E\circ f^1=\iota_{E^0}\circ f^0\circ s_G=\iota_{F^0}\circ g^0\circ s_G
=\iota_{F^0}\circ s_F\circ g^1\,,\nonumber\\
&\iota_{E^0}\circ t_E\circ f^1=\iota_{E^0}\circ f^0\circ t_G=\iota_{F^0}\circ g^0\circ t_G
=\iota_{F^0}\circ t_F\circ g^1\,,
\end{align}
which in turn follow, respectively, from the aforementioned commutativities of the above diagrams. Better still, the universal property of the pushout guarantees 
that both $s_{\coprod}$ and $t_{\coprod}$ are uniquely determined, and they render, respectively, the bottom subdiagrams of both diagrams commutative.

\begin{definition}
We call the graph 
\begin{equation*}
E\underset{G}{\amalg}F:=
\left(E^0\underset{G^0}{\amalg}F^0,E^1\underset{G^1}{\amalg}F^1,s_\amalg,t_\amalg\right)
\end{equation*} 
the \emph{pushout graph} of
$E\stackrel{(f^0,f^1)}{\longleftarrow} G\stackrel{(g^0,g^1)}{\longrightarrow} F$.
\end{definition}
\noindent 
It is straightforward to show that $E{\amalg}_GF$ is indeed a pushout in the category {\rm OG} of graphs and graph homomorphisms.

A pushout in the category {\rm OG} might not be a pushout in the admissible category of graphs. 
Therefore, we need the following result:
\begin{lemma}\label{admpush}
Let the diagram
\begin{equation*}
\xymatrix{
&
P
&\\
E
\ar[ur]^{(\iota^0_E,\iota^1_E)}& & 
F
\ar[ul]_{(\iota^0_F,\iota^1_F)}\\
&
G
\ar[ur]_{(g^0,g^1)}\ar[ul]^{(f^0,f^1)}&
}
\end{equation*}
be a one-injective pushout diagram in the category {\rm OG} of graphs and graph homomorphisms. 
If $(f^0,f^1)$ and $(g^0,g^1)$ are proper, regular, and
satisfy the target-bijectivity condition, then the same is true, respectively, for $(\iota_E^0,\iota_E^1)$ and $(\iota_F^0,\iota_F^1)$.
\end{lemma}
\begin{proof}
Let $(f^0,f^1)$ be injective and regular and let $(g^0,g^1)$ be proper and regular. Then, by Lemma~3.6, both $(\iota_E^0,\iota_E^1)$ and $(\iota_F^0,\iota_F^1)$ are proper, and, by the pushout property, $(\iota_F^0,\iota_F^1)$ is injective. Next, to prove the regularity of $(\iota_F^0,\iota_F^1)$, note that every infinite emitter in $F^0$ stays an infinite emitter in its image under $\iota_F^0$. Now, consider a~sink $v\in F^0$ and suppose that 
$w:=\iota_F^0(v)$ is regular. As $\iota_F^0$ is injective and $v$ is a sink, any edge emitted from $w$ must be in the image of~$\iota_E^1$. Hence, $w$ is in the image of $\iota_E^0$, so $v=g^0(u)$ for some $u\in G^0$. Then $u$ has to be a sink because, if there is an edge $a\in G^1$ 
such that $s_G(a)=u$, then 
\begin{equation}
v=g^0(u)=g^0(s_G(a))=s_E(g^1(a)),
\end{equation}
which is impossible. 
Furthermore, observe that $w=\iota_F^0(g^0(u))=\iota_E^0(f^0(u))$. The vertex $f^0(u)$ cannot be an infinite emitter because 
$\iota^1_E$ is proper and $w$ is regular. So suppose that $f^0(u)$ is a sink. Then, since $g^0(u)=v$ is also a~sink, this would again 
contradict the regularity of $w$, so $f^0(u)$ is regular, which contradicts the regularity of 
$(f^0,f^1)$ because $u$ is a sink.

Next, suppose that $(\iota_E^0,\iota_E^1)$ is not regular, i.e.\ that there is a vertex $v\in E^0\setminus{\rm reg}(E)$ such that 
$w:=\iota_E(v)\in{\rm reg}(P)$. Since $\iota_E$ is proper, the vertex $v$ cannot be an infinite emitter. Suppose that $v$ is a sink. If 
$v\notin f^0(G^0)$, then, $w\notin\iota_F^0(F^0)$, so we get a contradiction because $\iota_E^0$ is the identity when restricted to $E^0\setminus f^0(G^0)$. If 
there is a vertex $u\in G^0$ such that $f^0(u)=v$, then, arguing as before, $u$ is  a sink. Hence, as 
$w=\iota^0_E(f^0(u))=\iota^0_F(g^0(u))$, we get a contradiction with regularity of 
$(\iota_F^0\circ g^0,\iota_F^1\circ g^1)$.

Let us now prove that $(\iota_F^0,\iota_F^1)$ and $(\iota_E^0,\iota_E^1)$ satisfy the target-bijectivity condition. Take $x\in P^1$. 
Since $(\iota_F^0,\iota^1_F)$ is injective it suffices to exclude the two possibilities:
\begin{enumerate}
\item[(T1)] $(\iota_F^1)^{-1}(x)=\emptyset$ and $(\iota_F^0)^{-1}(t_P(x))=\{v\}$ for some $v\in F^0$,
\item[(T2)] $(\iota_F^1)^{-1}(x)=\{e\}$ for some $e\in F^1$ and $(\iota_F^0)^{-1}(t_P(x))=\emptyset$.
\end{enumerate}
Suppose that the condition (T1) is satisfied, so $\iota_F^0(v)=t_P(x)$. Since $P^1=\iota_E^1(E^1)\cup\iota_F^1(F^1)$, 
we infer that $x\in\iota_E^1(E^1)\setminus\iota_F^1(F^1)$. 
Hence, there is an edge $y\in E^1\setminus f^1(G^1)$ such that $\iota^1_E(y)=x$. Note that 
$\iota_F^0(v)=t_P(x)=t_P(\iota_E^1(y))=\iota_E^0(t_E(y))$.
Therefore, there is a vertex $u\in G^0$ such that $f^0(u)=t_E(y)$. Due to the target-surjectivity of $(f^0,f^1)$, we 
get an edge $a\in (f^1)^{-1}(y)$ such that $t_G(a)=u$. However, $\iota^1_F(g^1(a))=\iota^1_E(f^1(a))=\iota^1_E(y)=x$,
which contradicts $(\iota_F^1)^{-1}(x)=\emptyset$. Next, suppose that the condition (T2) is satisfied, so
$\iota_F^1(e)=x$ and $\iota_F^0(t_F(e))=t_P(\iota_F^1(e))=t_P(x)$, which contradicts 
$(\iota_F^0)^{-1}(t_P(x))=\emptyset$.

Finally, we prove that $(\iota_E^0,\iota_E^1)$ satisfies the target-bijectivity condition. 
Take any $x\in P^1$ and consider the following three cases:

\emph{Case 1:}
If $x\in\iota_F^1(F^1)\setminus \iota_E^1(E^1)$, then $(\iota^1_E)^{-1}(x)=\emptyset$ and there is an edge 
$e\in F^1\setminus g^1(G^1)$ such that $\iota_F^1(e)=x$. Suppose that there is a vertex $v\in E^0$ such that 
$\iota_E^0(v)=t_P(x)$. Then
\begin{equation}
\iota_F^0(t_F(e))=t_P(\iota_F^1(e))=t_P(x)=\iota_E^0(v),
\end{equation}
which implies that there is a vertex $u\in G^0$ such that $g^0(u)=t_F(e)$. Due to the target-surjectivity of 
$(g^0,g^1)$, there is an edge $a\in (g^1)^{-1}(e)$ such that $t_G(a)=u$. In turn, this implies that 
\begin{equation}
\iota_E^1(f^1(a))=\iota_F^1(g^1(a))=\iota_F^1(e)=x,
\end{equation}
which contradicts $(\iota_E^1)^{-1}(x)=\emptyset$.

\emph{Case 2:}
Let $x\in\iota_E^1(E^1)\cap\iota_F^1(F^1)$. Then we have the following commutative diagram:
\begin{equation}
\xymatrix{
(\iota_E^1)^{-1}(x) \ar[r]^{t_E} & (\iota_E^0)^{-1}(t_P(x))\\
(f^1)^{-1}((\iota_E^1)^{-1}(x)) \ar[u]^{f^1}_{\cong} \ar[r]_{t_G}^{\cong} 
& (f^0)^{-1}((\iota_E^0)^{-1}(t_P(x))) \ar[u]^{f^0}_{\cong}.
}
\end{equation}
Indeed, since $(\iota^1_E)^{-1}(x)\subseteq f^1(G^1)$ and $(\iota_E^0)^{-1}(t_P(x))\subseteq f^0(G^0)$, 
the two arrows going upwards are bijections.
To see that the bottom arrow is also a bijection, observe that 
\begin{gather}
(f^1)^{-1}((\iota_E^1)^{-1}(x))=(\iota_E^1\circ f^1)^{-1}(x)=(\iota_F^1\circ g^1)^{-1}(x),
\nonumber\\
(f^0)^{-1}\big((\iota_E^0)^{-1}(t_P(x))\big)=(\iota_E^0\circ f^0)^{-1}(t_P(x))=(\iota_F^0\circ g^0)^{-1}(t_P(x)).
\end{gather}
Now, as both $(\iota_F^0,\iota_F^1)$ and $(g^0,g^1)$ satisfy that target-bijectivity condition, so does their composition,
whence we infer the bijectivity of the bottom arrow. The desired bijectivity of  the top arrow follows now from the commutativity
of the diagram.

\emph{Case 3:} 
If $x\in\iota_E^1(E^1)\setminus\iota_F^1(F^1)$, then $(\iota_E^1)^{-1}(x)=\{y\}$ for some $y\in E^1$ because $\iota_E^1$ is 
injective when restricted to $E^1\setminus f^1(G^1)$. Consequently, $t_E(y)\in (\iota_E^0)^{-1}(t_P(x))$. Suppose that 
$(\iota_F^0)^{-1}(t_P(x))\neq\emptyset$. Then there is a vertex $w\in F^0$ such that $\iota_F^0(w)=t_P(x)$. It follows that 
\begin{equation}
\iota_E^0(t_E(y))=t_P(\iota^1_E(y))=t_P(x)=\iota_F^0(w),
\end{equation}
 which means that there is a vertex $u\in G^0$ such that 
$f^0(u)=t_E(y)$. By the target-surjectivity of $(f^0,f^1)$, there exists an edge $a\in (f^1)^{-1}(y)$. Therefore, 
\begin{equation}
x=\iota_E^1(f^1(a))=\iota_F^1(g^1(a)), 
\end{equation}
which contradicts our assumption.
\end{proof}

\subsection{The covariant \boldmath$FP$ functor}
For the purposes of our study of path algebras, we consider the pushout diagram in the category of sets and maps:
\begin{equation}
\xymatrix{
&
FP(E)\underset{FP(G)}{\amalg}FP(F)
&\\
FP(E)
\ar[ur]^{\iota_{FP(E)}}& & 
FP(F).
\ar[ul]_{\iota_{FP(F)}}\\
&
FP(G)
\ar[ur]_{g}\ar[ul]^{f}&
}
\end{equation}
This leads to the following natural question: 
Under which assumptions does the covariant functor $FP$ from the category $\mathrm{OG}$
of graphs and graph homomorphisms to the category of sets and maps commute with pushouts:
\begin{equation}
FP(E)\underset{FP(G)}{\amalg}FP(F)=FP\left(E\underset{G}{\amalg}F\right)\text{?}
\end{equation}
The answer is as follows:
\begin{lemma}\label{answer}
Let $E\stackrel{(f^0,f^1)}{\longleftarrow} G\stackrel{(g^0,g^1)}{\longrightarrow} F$ 
be graph homomorphisms.
Then there exists a natural map
\begin{equation*}
h\colon FP(E)\underset{FP(G)}{\amalg}FP(F)\longrightarrow FP\left(E\underset{G}{\amalg}F\right).
\end{equation*}
Moreover,  if the graph homomorphisms are such that
\vspace*{-2mm}
\begin{enumerate}
\item
both $f^0$ and $g^0$ are injective (vertex injectivity),
\item
$t_\amalg(x)=s_\amalg(y)\;\Rightarrow\;(x,y\in\iota_{E^1}(E^1)\text{ or }x,y\in\iota_{F^1}(F^1))$
(one color),
\end{enumerate}
then the natural map $h$ is bijective.
\end{lemma}
\begin{proof}
The natural map $FP(E)\amalg FP(F)\rightarrow FP(E\amalg_GF)$ exists because $FP$ is a functor. 
This map descends to $h$ by the universal
property of pushouts in the category of sets and maps. Since graph homomorphisms preserve the length
of paths, we obtain the decomposition
\begin{equation}
FP(E)\underset{FP(G)}{\amalg}FP(F)=\coprod_{n\in\mathbb{N}} FP_n(E)\underset{FP_n(G)}{\amalg}FP_n(F).
\end{equation}
Hence,  we can write $h$ on elements as follows:
\begin{equation}
h([p]):=\begin{cases}
[p] &\text{for } p\in \left(E^0\amalg F^0\right)\cup \left(E^1\amalg F^1\right),\\
([a_1],\ldots,[a_n]) &\text{for } p:=(a_1,\ldots,a_n)\in FP(E)\amalg FP(F),\\
&a_i\in E^1\amalg F^1,\;1\leq i\leq n,\;n\in\mathbb{N}\setminus\{0,1\}.
\end{cases}
\end{equation}

Now, using the two assumptions, we will define the inverse of~$h$. For starters, we put 
\begin{equation}\label{zeroone}
h^{-1}([p]):=[p]\quad
\text{when}\quad p\in \left(E^0\amalg F^0\right)\cup \left(E^1\amalg F^1\right).
\end{equation}
Next, let us take $([a_1],\ldots,[a_n])\in FP(E\amalg_GF)$
 with 
\begin{equation}
[a_i]\in E^1\underset{G^1}{\amalg}F^1,\; 1\leq i\leq n,\;
 n\in\mathbb{N}\setminus\{0,1\}. 
\end{equation}
Since $t_\amalg([a_i])=s_\amalg([a_{i+1}])$ for all $1\leq i\leq n-1$,
from the one-color condition we conclude that there exist $b_i\in E^1\amalg F^1$, $1\leq i\leq n$, such that $[b_i]=[a_i]$ for all $i$, and
\begin{equation}
\forall~1\leq i\leq n:~b_i\in E^1\quad\text{or}\quad\forall~1\leq i\leq n:~b_i\in F^1.
\end{equation}
Hence, $[t_E(b_i)]=[s_E(b_{i+1})]$ or $[t_F(b_i)]=[s_F(b_{i+1})]$. We can now apply the vertex-injectivity condition to infer that
$t_E(b_i)=s_E(b_{i+1})$ or $t_F(b_i)=s_F(b_{i+1})$. 

Furthermore, if $(b'_1,\ldots,b'_n)\in FP(E)$ or $(b'_1,\ldots,b'_n)\in FP(F)$, then 
\begin{equation}\label{implication}
\big{(}\forall~1\leq i\leq n:~[b_i]=[b'_i]\big{)}\quad \Longrightarrow\quad [(b_1,\ldots,b_n)]=[(b'_1,\ldots,b'_n)]
\end{equation}
Indeed, let $c^1_i,\ldots,c^{m}_i\in G^1$ be  sequences such that
\begin{gather}
\begin{matrix}
\forall\;1\leq i\leq n\colon\phantom{xxxxxxxxxx}\\
 f^1(c^1_i)=b_i\text{ and } g^1(c^1_i)=:b^1_i
\end{matrix}
\;,\ldots,\;
\begin{cases}
g^1(c^{m}_i)=:b^{m-1}_i\text{ and } f^1(c^{m}_i)=b'_i&\text{if $m$ is even},\\
f^1(c^{m}_i)=:b^{m-1}_i\text{ and } g^1(c^{m}_i)=b'_i&\text{if $m> 1$ is odd},
\end{cases}\nonumber\\
\text{or}\nonumber\\
\begin{matrix}
\forall\;1\leq i\leq n\colon\phantom{xxxxxxxxxx}\\
 g^1(c^1_i)=b_i\text{ and } f^1(c^1_i)=:b^1_i
\end{matrix}
\;,\ldots,\;
\begin{cases}
g^1(c^{m}_i)=:b^{m-1}_i\text{ and } f^1(c^{m}_i)=b'_i&\text{if $m>1$ is odd},\\
f^1(c^{m}_i)=:b^{m-1}_i\text{ and } g^1(c^{m}_i)=b'_i&\text{if $m$ is even}.
\end{cases}
\end{gather}
Note that, if the primed and unprimed edges are in the same graph, then the parity of all $c$-sequences is always even, and if they are in different graphs, then the parity is always odd. Therefore, although for each $i$ the length of the sequence $c^1_i,\ldots,c^{m}_i$ might be
different, we can always choose the longest such a sequence and extend shorter sequences by the constant extrapolation by an even number of elements. Next, we need to prove that $t_G(c^j_i)=s_G(c^j_{i+1})$. 
Depending on the parity of $j$ and the
above alternative between $E$ and $F$, we either have 
\begin{gather}
[f^1(c^j_i)]=[b_i]\quad\text{and}\quad [f^1(c^j_{i+1})]=[b_{i+1}]\nonumber\\
\text{or}\nonumber\\
[g^1(c^j_i)]=[b_i]\quad\text{and}\quad[g^1(c^j_{i+1})]=[b_{i+1}].
\end{gather}
 In the former case, by the vertex injectivity, we obtain
\begin{equation}\label{eq}
t_E(f^1(c^j_i))=t(b_i)\quad\text{and}\quad s_E(f^1(c^j_{i+1}))=s(b_{i+1}),
\end{equation}
where $t$ and $s$ are, respectively, $t_E$ and $s_E$, or $t_F$ and $s_F$, depending on whether
$b_i\in E^1$ or $b_i\in F^1$. It follows from \eqref{eq} that
\begin{equation}
f^0(t_G(c^j_i))=t_E(f^1(c^j_i))=t(b_i)=s(b_{i+1})=s_E(f^1(c^j_{i+1}))=f^0(s_G(c^j_{i+1})),
\end{equation}
so, from the injectivity of $f^0$, we get $t_G(c^j_i)=s_G(c^j_{i+1})$, as needed. In the latter case,
the reasoning is completely analogous but uses the injectivity of $g^0$ instead of the injectivity of~$f^0$.
Thus we have shown that $(c^j_1,\ldots,c^j_n)\in FP(G)$, $1\leq j\leq m$, is a sequence of paths
 implementing the desired equivalence relation between $(b_1,\ldots,b_n)$ and $(b'_1,\ldots,b'_n)$.

Finally, it follows from \eqref{implication} that we can define 
\begin{equation}
h^{-1}(([a_1],\ldots,[a_n])):=[(b_1,\ldots,b_n)]
\end{equation} 
Combining it with \eqref{zeroone} gives us
a map 
\begin{equation}
h^{-1}:FP\left(E\underset{G}{\amalg}F\right)\longrightarrow FP(E)\underset{FP(G)}{\amalg}FP(F),
\end{equation} 
which is, clearly, the inverse of~$h$. 
\end{proof}

\subsection{From pushouts to pullbacks}
Let us consider the contravariant
functors $\mathrm{Map}(\cdot,K)$ and $\mathrm{Map}_f(\cdot,K)$, where $K$ is a~non-empty
 set with a chosen element $0\in K$, ${\rm Map}$ stands for all maps, and ${\rm Map}_f$ denotes finitely supported maps (all but finitely many elements
are mapped to~0). The first functor is
a contravariant functor from the category of sets and maps to the category of sets and maps:
\begin{equation*}
X\longmapsto \mathrm{Map}(X,K),\quad (X\stackrel{f}{\to}Y)\longmapsto
 \left(\mathrm{Map}(Y,K)\stackrel{f^*}{\to}\mathrm{Map}(X,K)\right),\quad f^*(F):=F\circ f.
\end{equation*}
Much in the same way, the second functor is
a contravariant functor from the category of sets and finite-to-one maps to the category of sets and maps:
\begin{equation*}
X\longmapsto \mathrm{Map}_f(X,K),\quad (X\stackrel{f}{\to}Y)\longmapsto
 \left(\mathrm{Map}_f(Y,K)\stackrel{f^*}{\to}\mathrm{Map}_f(X,K)\right),\quad f^*(F):=F\circ f.
\end{equation*}
Here $\mathrm{Map}_f(X,K):=\{F\in\mathrm{Map}(X,K)\;|\;|F^{-1}(K\setminus\{0\})|<\infty \}$,
and
\begin{equation}
|(F\circ f)^{-1}(K\setminus\{0\})|=|f^{-1}(F^{-1}(K\setminus\{0\}))|<\infty
\end{equation}
because a finite union of finite sets is a finite set.

We have the following elementary lemmas whose routine proofs we omit.
\begin{lemma}\label{lemma73}
Let $K$ be any set. Then $\mathrm{Map}(\cdot,K)$ is a contravariant functor from the category of sets
and maps to the category of sets and maps transforming pushouts to pullbacks.
\end{lemma}
\noindent If we restrict to finite-to-one maps, we get a contravariant functor ${\rm Map}_f(\cdot,K)$. 
Furthermore, using Lemma~\ref{oneinjective}, we have the following result.
\begin{lemma}\label{finitepushtopull}
Let $K$ be a non-empty set with a chosen element $0\in K$,  let $f$ be an injective map, and let
\begin{equation*}
\xymatrix{
&
X\underset{Z}{\amalg}Y
&\\
X
\ar[ur]^{\iota_X}& & 
Y
\ar[ul]_{\iota_Y}\\
&
Z
\ar[ur]_{g}\ar[ul]^{f}&
}
\end{equation*}
be a pushout diagram in the category of sets and finite-to-one maps.
Then  the contravariant functor $\mathrm{Map}_f(\cdot,K)$ transforms the above pushout diagram
into the  pullback dia\-gram
\begin{equation*}
\xymatrix{
&
\mathrm{Map}_f\left(X\underset{Z}{\amalg}Y,K\right) 
\ar[dl]_{\iota_Y^*}\ar[dr]^{\iota_X^*}
&\\
\mathrm{Map}_f(X,K)
\ar[dr]_{f^{*}}& & 
\mathrm{Map}_f(Y,K)
\ar[dl]^{g^{*}}\\
&
\mathrm{Map}_f(Z,K).
&
}
\end{equation*}
in the category of sets and maps such that its left defining morphism is surjective.
\end{lemma}

%%%%%%%%%%%%%%%%%%%%%%%%%%%%%%%%%%%%%%%%%%%%%%%%%%%%%%%%%%%%%%

\section{Graph algebras as contravariant functors}

\noindent To alleviate the notation, from now on we will write $e_1\ldots e_n$ for a path $(e_1,\ldots,e_n)$. Furthermore, if $p=e_1\ldots e_n$, $q=f_1\ldots f_m$, and the end of $e_n$ is the beginning of $f_1$, then $pq=e_1\ldots e_n f_1\ldots f_m$.

\subsection{Path algebras}

Let $k$ be a field, $E$ be any non-empty graph, and $FP(E)$ be the set of all its finite paths in $E$. Consider the vector space
\begin{equation}
kE:={\rm Map}_f(FP(E),k),
\end{equation}
where ${\rm Map}_f(FP(E),k)$ is the vector space of all finitely supported functions from $FP(E)$ to $k$ in which the addition and scalar multiplication are 
pointwise. Then the set of functions $\{\chi_p\}_{p\in FP(E)}$ given by
\begin{equation}
\chi_p(q)=\begin{cases}1 &\text{for }p=q,\\0 & \text{otherwise},\end{cases}
\end{equation}
is a linear basis of $kE$. 
By checking the associativity, one can prove that the formulas
\begin{equation}
m:kE\times kE\longto kE,\qquad m(\chi_p,\chi_q):=\begin{cases}\chi_{pq} & \text{if }t(p)=s(q),\\0 & \text{otherwise},\end{cases}
\end{equation}
define a multiplication on~$kE$ rendering it a $k$-algebra. 
\begin{definition}{{\rm (\cite[Definition~II.1.2]{ass06})}}
Let $E$ be a non-empty graph. 
The above constructed algebra $(kE,+,0,m)$ is called the {\em path algebra} of~$E$.
If $E=\emptyset$, then $kE:=0$.
\end{definition}

Let $\mathrm{KA}$ denote the category of algebras over $k$ together with algebra homomorphisms, 
and let $\mathrm{UKA}$ denote the category of unital 
algebras over $k$ together with unital algebra homomorphisms.
\begin{lemma}\label{contrafp}
The assignment
\begin{gather}
\mathrm{Obj}({\rm POG})\ni E\stackrel{\mathrm{Map}_f}{\longmapsto}kE\in\mathrm{Obj}(\mathrm{KA}),\nonumber\\
\mathrm{Mor}({\rm POG})\ni ((f^0,f^1)\colon E\to F)
\stackrel{\mathrm{Map}_f}{\longmapsto}(f^*\colon kF\to kE)\in\mathrm{Mor}(\mathrm{KA}),\nonumber\\
kF\ni \chi_p\stackrel{f^*}{\longmapsto}\sum_{q\in f^{-1}(p)}\chi_q\in kE\,,\label{sum}
\end{gather}
where $f\colon FP(E)\to FP(F)$ is the map induced by $(f^0,f^1)$, defines a contravariant functor. Furthermore, the same assignment 
restricted to the subcategory 
given by graphs with finitely many vertices yields a contravariant functor to the category~$\mathrm{UKA}$.
\end{lemma}
\begin{proof}
Since $f:FP(E)\to FP(F)$ is finite to one, the sum in \eqref{sum} is well defined.
 Furthermore, for any $r\in FP(E)$ and $p\in FP(F)$, we have:
\begin{equation}
f^*(\chi_p)(r):=\left(\sum_{q\in f^{-1}(p)}\chi_q\right) (r)=
\left(\sum_{q\in f^{-1}(p)}\delta_{q,r}\right)=\delta_{p,f(r)}=\chi_p(f(r))=
(\chi_p\circ f)(r).
\end{equation} 
Hence, $f^*$ is the pullback linear map
\begin{equation}
f^*\colon \mathrm{Map}_f(FP(F),k)\ni\alpha\longmapsto\alpha\circ f\in \mathrm{Map}_f(FP(E),k).
\end{equation}
The contravariance  is obvious because
$(f\circ g)^*=g^*\circ f^*$. 

To check that it is an algebra homomorphism, using the fact that $f$ preserves the length of paths,
 we compute:
 %\pagebreak
\begin{align}
f^*(\chi_p\chi_q)&=\delta_{t_F(p),s_F(q)}f^*(\chi_{pq})=\delta_{t_F(p),s_F(q)}\sum_{r\in f^{-1}(pq)}\chi_r
\\
&=\sum_{\substack{r_1\in f^{-1}(p)\\ r_2\in f^{-1}(q)}}\delta_{t_E(r_1),s_E(r_2)}\chi_{r_1r_2}
=\sum_{r_1\in f^{-1}(p)}\chi_{r_1}\sum_{r_2\in f^{-1}(q)}\chi_{r_2}=f^*(\chi_p)f^*(\chi_q). \nonumber
\end{align}
Here, in the third step, we used the implication
\begin{equation}
t_F(p)\neq s_F(q)\quad\Longrightarrow\quad \forall\;r_1\in f^{-1}(p),\, r_2\in f^{-1}(q)\colon t_E(r_1)\neq
s_E(r_2).
\end{equation}
Finally, the unitality of $f^*$ for unital path algebras follows from the fact that $f^{-1}(F^0)=E^0$:
\begin{equation}
f^*(1)=f^*\left(\sum_{v\in F^0}\chi_v\right)=\sum_{w\in f^{-1}(F^0)}\chi_w
=\sum_{w\in E^0}\chi_w=1\,.
\vspace*{-7mm}\end{equation}
\end{proof}

\subsection{Leavitt path algebras}
Let $E=(E^0,E^1,s_E,t_E)$ be a graph. The {\em extended graph} $\bar{E}:=(\bar{E}^0, \bar{E}^1,s_{\bar{E}},t_{\bar{E}})$
of the graph $E$ is given as follows:
\begin{gather}
\bar{E}^0:=E^0,\quad \bar{E}^1:=E^1\sqcup (E^1)^*,\quad (E^1)^*:=\{e^*~|~e\in E^1\},
\nonumber\\
\forall\; e\in E^1:\quad s_{\bar{E}}(e):=s_E(e),\quad t_{\bar{E}}(e):=t_E(e),
\nonumber\\
\forall\; e^*\in (E^1)^*:\quad s_{\bar{E}}(e^*):=t_E(e),\quad t_{\bar{E}}(e^*):=s_E(e).
\end{gather}
Observe that every graph homomorphism $(f^0,f^1)\colon E\to F$ can be extended to a graph homomorphism 
$(\bar{f}^0,\bar{f}^1)\colon\bar{E}\to\bar{F}$ in the following way:
\begin{align}
\bar{f}^0(v)&:=f^0(v),\; v\in E^0=\bar{E}^0,\nonumber\\ 
\bar{f}^1(e)&:=f^1(e),\; e\in E^0,\nonumber\\ 
\bar{f}^1(e^*)&:=f^1(e)^*,\; e^*\in (E^1)^*.
\end{align}
Also, if $e_1\ldots e_n\in FP_n(E)$, we define
\begin{equation}
(e_1\ldots e_n)^*:=e^*_n\ldots e_1^*\in FP(\bar{E}).
\end{equation}

We state the following straightforward result without a proof.
\begin{lemma}\label{extlem}
The assignment
\begin{equation*}
E\longmapsto \bar{E},\qquad (f^0,f^1)\longmapsto (\bar{f}^0,\bar{f}^1),
\end{equation*}
defines a covariant endofunctor of the category {\rm OG}
 of graphs and graph homomorphisms. Furthermore, it restricts to an endofunctor of the 
category {\rm POG} of graphs and proper graph homomorphisms.
\end{lemma}

\begin{definition}
The {\em Leavitt path algebra} $L_k(E)$ of a graph $E$ is the quotient of the path algebra $k\bar{E}$ of the extended graph 
$\bar{E}$ by the ideal generated by the union of the following sets:
\begin{enumerate}
\item[{\rm (CK1)}] $\{\chi_{e^*}\chi_f-\delta_{e,f}\chi_{t_E(e)}~|~e,f\in E^1\}$,
\item[{\rm (CK2)}] $\Big{\{}\chi_v-\sum_{e\in s^{-1}_E(v)}\chi_e\chi_{e^*}~\Big{|}~v\in{\rm reg}(E)\Big{\}}$.
\end{enumerate}
\end{definition}
\noindent
The algebra $L_k(E)$ is isomorphic with the universal algebra generated by the elements $\chi_v$, $\chi_e$, $\chi_{e^*}$, $v\in E^0$, 
$e\in E^1$, subject to the relations (CK1) and (CK2) and the standard path-algebraic relations $\chi_v\chi_w=\delta_{v,w}\chi_v$ 
and $\chi_{s_E(e)}\chi_e=\chi_e\chi_{t_E(e)}$. The algebra $L_k(E)$ is $\mathbb{Z}$-graded by the lengths of paths, where edges from 
$(E^1)^*$ have length~$-1$. For $k=\mathbb{C}$, the $\mathbb{Z}$-grading is equivalent to the $U(1)$-action $\gamma$ on 
$L_k(E)$, called the {\em gauge action}, 
given by
\begin{equation}\label{algaction}
\gamma_z([\chi_v]):=[\chi_v],\; \gamma_z([\chi_e]):=z[\chi_e],\; \gamma_z([\chi_{e^*}]):=\bar{z}[\chi_{e^*}],
\; z\in U(1),~v\in E^0,~e\in E^1\,.
\end{equation}

Let $\mathrm{ZKA}$ denote the category of $\mathbb{Z}$-graded algebras over $k$ together with algebra homomorphisms 
preserving the $\mathbb{Z}$-grading, and $\mathrm{ZUKA}$ denote its subcategory of unital $\mathbb{Z}$-graded algebras and 
unital homomorphisms preserving the $\mathbb{Z}$-grading. The next theorem exploits the contravariant functoriality of the 
Leavitt-path-algebra construction when restricted to the admissible category of graphs. %{\rm CRTBPOG}.
\begin{theorem}\label{contralthm}
 The assignment
\begin{gather*}
\mathrm{Obj}({\rm CRTBPOG})\ni E\stackrel{}{\longmapsto}L_k(E)\in\mathrm{Obj}(\mathrm{ZKA}),\nonumber\\
\mathrm{Mor}({\rm CRTBPOG})\ni ((f^0,f^1)\colon E\to F)
\stackrel{}{\longmapsto}\big({f}^*_L\colon L_k(F)\to L_k(E)\big)\in\mathrm{Mor}(\mathrm{ZKA}),
\nonumber\\
L_k(F)\ni [\chi_p]\stackrel{{f}^*_L}{\longmapsto}\sum_{q\in \bar{f}^{-1}(p)}[\chi_q]\in L_k(E)\,,\label{lsum}
\nonumber
\end{gather*}
where $\bar{f}\colon FP(\bar{E})\to FP(\bar{F})$ is the map induced by $(f^0,f^1)$, defines a contravariant functor. Furthermore, 
the same assignment restricted to the subcategory 
given by graphs with finitely many vertices yields a contravariant functor to the category~$\mathrm{ZUKA}$. 
\end{theorem}
\begin{proof}
It follows from Lemma~\ref{contrafp} combined 
with Lemma~\ref{extlem} that we have a~contravariant functor
\begin{gather}
\mathrm{Obj}({\rm POG})\ni E\stackrel{}{\longmapsto}k\bar{E}\in\mathrm{Obj}(KA),
\nonumber\\
\mathrm{Mor}({\rm POG})\ni ((f^0,f^1)\colon E\to F)
\stackrel{}{\longmapsto}\big(\bar{f}^*\colon k\bar{F}\to k\bar{E}\big)\in\mathrm{Mor}(KA),
\nonumber\\
k\bar{F}\ni \chi_p\stackrel{\bar{f}^*}{\longmapsto}\sum_{q\in \bar{f}^{-1}(p)}\chi_q\in k\bar{E} \,.
\end{gather}
To complete the proof of the first statement, it suffices to show that $\bar{f}^*$ descends to a $\mathbb{Z}$-graded algebra homomorphism 
$L_k(F)\to L_k(E)$ when we restrict to the admissible category of graphs.
 
First, we have to demonstrate that
\begin{equation}
\forall\;f\in\mathrm{Mor}({\rm CRTBPOG})\;\forall\;x,y\in F^1\colon
 [\bar{f}^*(\chi_{x^*})][\bar{f}^*(\chi_y)]=\delta_{x,y}[\bar{f}^*(\chi_{t_F(x)})],
\end{equation}
which is equivalent to
\begin{equation}\label{ck1cohn}
\forall\;f\in\mathrm{Mor}({\rm CRTBPOG})\;\forall\;x,y\in F^1\colon 
\sum_{\substack{e_x\in f^{-1}(x)\\ e_y\in f^{-1}(y)}}[\chi_{e_x^*}][\chi_{e_y}]
=\delta_{x,y}\hspace*{-4mm}\sum_{v\in f^{-1}(t_F(x))}[\chi_v].
\end{equation}
For $x\neq y$, we have $f^{-1}(x)\cap f^{-1}(y)=\emptyset$, so \eqref{ck1cohn} clearly holds.
For $x=y$, using the target-bijectivity condition~\eqref{EU}, we compute:
\begin{equation}
\sum_{e_1,\,e_2\in f^{-1}(x)}[\chi_{e_1^*}][\chi_{e_2}]
=
\sum_{e\in f^{-1}(x)}[\chi_{e^*}][\chi_{e}]
 =
\sum_{e\in f^{-1}(x)}[\chi_{t_E(e)}]
 =
\sum_{v\in f^{-1}(t_F(x))}[\chi_{v}].
\end{equation}

Next, much as above, we need to show that
\begin{equation}\label{suffice}
\forall\;w\in\mathrm{reg}(F)\colon \sum_{\substack{x\in s_F^{-1}(w)\\e_1,e_2\in (f^1)^{-1}(x)}}[\chi_{e_1}][\chi_{e_2^*}]=\sum_{v\in (f^0)^{-1}(w)}[\chi_v].
\end{equation}
For starters, using \eqref{214}, we obtain
\begin{align}
\sum_{v\in (f^0)^{-1}(w)}[\chi_v]
&=\sum_{\substack{v\in (f^0)^{-1}(w)\\ e\in s_E^{-1}(v)}}[\chi_e][\chi_{e^*}]=
\sum_{e\in (f^0\circ s_E)^{-1}(w)}[\chi_e][\chi_{e^*}]
\nonumber\\ &=
\sum_{e\in (s_F\circ f^1)^{-1}(w)}[\chi_e][\chi_{e^*}]=
\sum_{\substack{x\in s_F^{-1}(w)\\ e\in (f^1)^{-1}(x)}}[\chi_e][\chi_{e^*}].
\end{align}
Next, using the target-injectivity condition \eqref{EU},  we compute the left-hand side of~\eqref{suffice}:
\begin{equation}
\sum_{\substack{x\in s_F^{-1}(w)\\ e_1,e_2\in (f^1)^{-1}(x)}}[\chi_{e_1}][\chi_{e_2^*}]=\sum_{\substack{x\in s_F^{-1}(w)\\ e\in (f^1)^{-1}(x)}}[\chi_{e}][\chi_{e^*}].
\end{equation}

Summarizing, we have proved that $\bar{f}^*$ descends to an algebra homomorphism\linebreak
 \mbox{$f^*_L\colon L_k(F)\to L_k(E)$}, which obviously preserves the $\mathbb{Z}$-grading because $\bar{f}$ preserves the lengths of paths. 
Finally, the functoriality of the assignment 
\begin{equation}
{\rm Mor}({\rm CRTBPOG})\ni (f^0,f^1)\longmapsto {f}^*_L\in {\rm Mor}(\mathrm{ZKA})
\end{equation}
 is immediate, and the unitality of ${f}^*_L$ for graphs with finitely many vertices follows from  the unitality of $\bar{f}^*$
under the same restriction.
\end{proof}

\begin{corollary}\label{injsurcor}
If $f:E\to F$ is an injective (surjective) morphism in the admissible category of graphs, then $f^*_L$ is surjective (injective).
\end{corollary}
\begin{proof}
By Theorem~\ref{contralthm}, the admissibility of $f$ implies the existence of a $\mathbb{Z}$-graded algebra homomorphism $f^*_L:L_k(F)\to L_k(E)$. If $f:E\to F$ is 
injective, then $[\chi_v]=f^*_L([\chi_{f(v)}])$ for all $v\in E^0$ and $[\chi_e]=f^*_L([\chi_{f(e)}])$ for all $e\in E^1$. Consequently, $f^*_L$ is surjective because is 
$L_k(E)$ is generated by $[\chi_v]$, $v\in E^0$, $[\chi_e]$ and $[\chi_{e^*}]$, $e\in E^1$. Vice versa, if $f:E\to F$ is surjective, then $f^{-1}(w)\neq\emptyset$ for all 
$w\in F^0$, so
\begin{equation}
\forall w\in F^0:\qquad f^*_L([\chi_w])=\sum_{v\in (f^0)^{-1}(w)}[\chi_v]\neq 0.
\end{equation}
Here the last step follows from the basis theorem~\cite[Corollary~1.5.12]{aasm17}. Now, from the graded uniqueness theorem for Leavitt path algebras~\cite[Theorem~2.2.15]{aasm17}, we conclude that $f^*_L$ is injective. 
\end{proof}

\subsection{Graph C*-algebras}
For basic facts about C*-algebras, we refer the reader to~\cite{j-d77}. Let us now consider the Leavitt path algebra construction in 
the case $k=\mathbb{C}$. One defines an anti-homomorphism ${}^*:L_\mathbb{C}(E)\to L_\mathbb{C}(E)$ by mapping the generators as follows: 
\begin{equation}
([\chi_v])^*:=[\chi_v],\quad ([\chi_e])^*:=[\chi_{e^*}],\quad ([\chi_{e^*}])^*:=[\chi_e],\quad v\in E^0,~e\in E^1\,.
\end{equation}
The above defined ${}^*$ operation turns $L_\mathbb{C}(E)$ into a complex $*$-algebra. Thus we arrive at the key definition.
\begin{definition}{{\rm (}\cite[Definition~5.2.1]{aasm17}{\rm )}}\label{graphc*}
Let $E$ be a graph. The {\em graph C*-algebra} $C^*(E)$ of $E$ is the universal C*-envelope of the complex $*$-algebra 
$L_\mathbb{C}(E)$.
\end{definition}

It is worth noting that, unlike for general universal C*-envelopes, for the graph C*-algebras the canonical $*$-homomorphism 
$L_\mathbb{C}(E)\to C^*(E)$ is injective (e.g., see~\cite[Theorem~5.2.9]{aasm17}). Better still, the gauge action \eqref{algaction}
extends to graph C*-algebras by
universality and continuity.
Note also that  Definition~\ref{graphc*} is equivalent to~\cite[Definition~1]{flr00} defining $C^*(E)$ as the universal C*-algebra 
generated by mutually orthogonal projections $P_v$, $v\in E^0$, and partial isometries $S_e$, $e\in E^1$, with mutually 
orthogonal ranges, satisfying
\begin{enumerate}
\item $S^*_eS_e=P_{t(e)}$ for all $e\in E^1$,
\item $P_v=\sum_{e\in s^{-1}(v)}S_eS_e^*$ for all $v\in {\rm reg}(E)$,
\item $S_eS_e^*\leq P_{s(e)}$ for all $e\in E^1$.
\end{enumerate}
In what follows, we will need the notaion $S_p:=S_{e_1}S_{e_2}\ldots S_{e_n}$ and $p^*:=e_n^*\ldots e_1^*$
 for a positive-length path $p=e_1\ldots e_n$, and $S_v:=P_v$ for a 0-length path~$v$. Note that this notation allows us to replace $[\chi_p]$ by $S_p$.

Let $\mathrm{GC^*\!A}$ denote the category of $U(1)$-C*-algebras together with $U(1)$-equivariant $*$-homo\-morphisms,
 and let $\mathrm{GUC^*\!A}$ denote the category of unital $U(1)$-C*-algebras and unital $U(1)$-equivariant $*$-homomorphisms.
The following C*-algebraic counterpart of Theorem~\ref{contralthm} is the discrete-topology case of Katsura's
\cite[Proposition~2.9]{t-k06}. As we were
unaware of the just cited result prior to obtaining our own version, we retain its complete and self-contained original proof,
which is routed via the algebraic constructions of path algebras and Leavitt path algebras absent in Katsura's work.
\begin{corollary}\label{corgc*}
The assignment
\begin{gather*}
\mathrm{Obj}({\rm CRTBPOG})\ni E\stackrel{}{\longmapsto}C^*(E)\in\mathrm{Obj}(\mathrm{GC^*\!A}),\nonumber\\
\mathrm{Mor}({\rm CRTBPOG})\ni ((f^0,f^1)\colon E\to F)
\stackrel{}{\longmapsto}\big(f^*_{C^*}\colon C^*(F)\to C^*(E)\big)\in\mathrm{Mor}(\mathrm{GC^*\!A}),
\nonumber\\
C^*(F)\ni S_p\stackrel{f^*_{C^*}}{\longmapsto}\sum_{q\in \bar{f}^{-1}(p)}S_q\in C^*(E)\,,%\label{lsum}
\nonumber
\end{gather*}
where $\bar{f}\colon FP(\bar{E})\to FP(\bar{F})$ is the map induced by $(f^0,f^1)$, defines a contravariant functor. Furthermore, 
the same assignment restricted to the subcategory 
given by graphs with finitely many vertices yields a contravariant functor to the category~$\mathrm{GUC^*\!A}$. 
\end{corollary}
\begin{proof}
From Theorem~\ref{contralthm}, in the case $k=\mathbb{C}$, we know that every 
$(f^0,f^1)\in{\rm Mor}(\mathrm{CRTBPOG})$ gives rise to 
a homomorphism ${f}^*_L\colon L_\mathbb{C}(F)\to L_\mathbb{C}(E)$ that preserves the $\mathbb{Z}$-grading
 coming from the path lengths. It is automatically a $*$-homomorphism because
\begin{equation}
{f}^*_L([\chi_p]^*)={f}^*_L([\chi_{p^*}])=\sum_{q^*\in\bar{f}^{-1}(p^*)}[\chi_{q^*}]=\left(\sum_{q\in\bar{f}^{-1}(p)}
[\chi_q]\right)^*={f}^*_L([\chi_p])^*.
\end{equation}
Since the grading is equivalent to the gauge action, ${f}^*_L$ is gauge equivariant. 
From~\cite[Theorem~4.4]{at-11}, we infer that ${f}^*_L$ extends to
 a unique $*$-homomorphism $f^*_{C^*}\colon C^*(F)\to C^*(E)$, which is also gauge equivariant by the continuity
of the gauge action and $*$-homomorphisms between C*-algebras. Finally, the unitality of $f^*_{C^*}$ for graphs with
finitely many vertices follows from the unitality of  ${f}^*_L$ under the same restriction.
\end{proof}
Furthermore, note that using~\cite[Theorem~5.2.12]{aasm17} instead of~\cite[Theorem~2.2.15]{aasm17}, 
we obtain a C*-algebraic version of Corollary~\ref{injsurcor}.
\begin{corollary}\label{c*injsur}
If $f:E\to F$ is an injective (surjective) morphism in the admissible category of graphs, then $f^*_{C^*}$ is surjective (injective).
\end{corollary}

\section{New types of admissible graph homomorphisms}
\noindent

\subsection{Generalized foldings}
For starters, let us observe that mapping the
disjoint union of a graph with its copy into the graph by identifying the same elements in two different copies is
 an admissible morphism 
inducing the diagonal map:
\begin{equation}\label{diagonal}
f\colon E\sqcup E \longrightarrow E \quad\leadsto\quad f^*_{C^*}\colon C^*(E)\ni a\longmapsto (a,a) \in C^*(E)\oplus C^*(E).
\end{equation}
In this section, we replace  disjoint unions of graphs by pushouts over an admissible subgraph. 

\begin{definition}\label{genfold}
Let $G$ be an admissible subgraph of a graph $F$.
A {\em generalized folding} is the graph homomorphism
\begin{equation}\label{foldeq}
g:F\underset{G}{\amalg}F\longrightarrow F
\end{equation}
given by the universal property of the pushout applied to the identity maps $F\to F$.
\end{definition}
\noindent It is clear that generalized foldings are morphisms in the admissible category of graphs and that Definition~\ref{genfold} 
generalizes Example~\ref{stallings}. It also includes~\eqref{diagonal} as a special case obtained by taking $G$ to be the empty graph. Let us now further exemplify Definition~\ref{genfold}.
\begin{example}\label{quantsphere}
{\rm 
For any $n\in\mathbb{N}$ and $q\in [0,1)$, the C*-algebra $C(S^{2n}_q)$ of the Hong--Szyma\'nski even quantum sphere $S_q^{2n}$ 
is isomorphic to the graph C*-algebra of the graph $L_{2n}$ (see~\cite[Section~5.1]{hs02}).

\begin{figure}[h]
\begin{tikzpicture}[auto,swap]
\tikzstyle{vertex}=[circle,fill=black,minimum size=3pt,inner sep=0pt]
\tikzstyle{edge}=[draw,->]
\tikzstyle{cycle1}=[draw,->,out=130, in=50, loop, distance=30pt]
   
\node[vertex] (0) at (0,0) {};
\node[vertex] (1) at (-0.5,-0.5) {};
\node[vertex] (2) at (0.5,-0.5) {};

\path (0) edge[cycle1] node[above] {} (0);
\path (0) edge[edge] node[above left] {} (1);
\path (0) edge[edge] node[above right] {} (2);
\end{tikzpicture}\qquad
\begin{tikzpicture}[auto,swap]
\tikzstyle{vertex}=[circle,fill=black,minimum size=3pt,inner sep=0pt]
\tikzstyle{edge}=[draw,->]
\tikzstyle{cycle1}=[draw,->,out=130, in=50, loop, distance=30pt]
   
\node[vertex] (0) at (0,0) {};
\node[vertex] (1) at (1,0) {};
\node[vertex] (2) at (0,-0.5) {};
\node[vertex] (3) at (1,-0.5) {};

\path (0) edge[cycle1] node {} (0);
\path (1) edge[cycle1] node {} (1);
\path (0) edge[edge] node {} (1);
\path (0) edge[edge] node {} (2);
\path (0) edge[edge] node {} (3);
\path (1) edge[edge] node {} (2);
\path (1) edge[edge] node {} (3);
\end{tikzpicture}
\caption{The graph $L_{2n}$ for $n=1$ and $n=2$.}
\end{figure}
\noindent Similarly, the C*-algebra $C(S^{2n-1}_q)$ of the Vaksman--Soibelman odd quantum sphere $S^{2n-1}_q$~\cite{vs90} 
is isomorphic to the graph C*-algebra of the graph $L_{2n-1}$ (see~\cite[Section~4.1]{hs02}).% (\cite{km23}). 

\begin{figure}[h]
\begin{tikzpicture}[auto,swap]
\tikzstyle{vertex}=[circle,fill=black,minimum size=3pt,inner sep=0pt]
\tikzstyle{edge}=[draw,->]
\tikzstyle{cycle1}=[draw,->,out=130, in=50, loop, distance=30pt]
   
\node[vertex] (0) at (0,1) {};

\path (0) edge[cycle1] node[above] {} (0);
\end{tikzpicture}\qquad
\begin{tikzpicture}[auto,swap]
\tikzstyle{vertex}=[circle,fill=black,minimum size=3pt,inner sep=0pt]
\tikzstyle{edge}=[draw,->]
\tikzstyle{cycle1}=[draw,->,out=130, in=50, loop, distance=30pt]
   
\node[vertex] (0) at (0,0) {};
\node[vertex] (1) at (1,0) {};

\path (0) edge[cycle1] node {} (0);
\path (1) edge[cycle1] node {} (1);
\path (0) edge[edge] node {} (1);
\end{tikzpicture}
\caption{The graph $L_{2n-1}$ for $n=1$ and $n=2$.}
\end{figure}
\noindent Finally, the C*-algebra $C(B^{2n}_q)$ of the Hong--Szyma\'nski even quantum ball $B^{2n}_q$ is isomorphic 
to the graph C*-algebra of $M_n$ (see~\cite[Section~3.1]{hs08} and~\cite{hs02}).

\begin{figure}[h]
\begin{tikzpicture}[auto,swap]
\tikzstyle{vertex}=[circle,fill=black,minimum size=3pt,inner sep=0pt]
\tikzstyle{edge}=[draw,->]
\tikzstyle{cycle1}=[draw,->,out=130, in=50, loop, distance=30pt]
   
\node[vertex] (0) at (0,0) {};
\node[vertex] (1) at (0,-0.5) {};

\path (0) edge[cycle1] node[above] {} (0);
\path (0) edge[edge] node[above left] {} (1);
\end{tikzpicture}\qquad
\begin{tikzpicture}[auto,swap]
\tikzstyle{vertex}=[circle,fill=black,minimum size=3pt,inner sep=0pt]
\tikzstyle{edge}=[draw,->]
\tikzstyle{cycle1}=[draw,->,out=130, in=50, loop, distance=30pt]
   
\node[vertex] (0) at (0,0) {};
\node[vertex] (1) at (1,0) {};
\node[vertex] (2) at (0.5,-0.5) {};

\path (0) edge[cycle1] node {} (0);
\path (1) edge[cycle1] node {} (1);
\path (0) edge[edge] node {} (1);
\path (0) edge[edge] node {} (2);
\path (1) edge[edge] node {} (2);
\end{tikzpicture}
\caption{The graph $M_{n}$ for $n=1$ and $n=2$.}
\end{figure}

It is clear that $L_{2n-1}$ is an admissible subgraph of $L_{2n}$ and that $L_{2n}=M_{n}\sqcup_{L_{2n-1}}M_{n}$, 
for every $n\in\mathbb{N}$ (e.g., see~\cite[Section~4.1]{hrt20}). Therefore, we can consider the generalized folding
\begin{equation}\label{foldsphere}
L_{2n}=M_{n}\underset{L_{2n-1}}{\amalg}M_{n}\longrightarrow M_n\,.
\end{equation}
Note that the admissible graph homomorphism~\eqref{foldsphere} induces a gauge-equivariant unital 
\mbox{$*$-ho}\-mo\-mor\-phism
\begin{equation}
C(B^{2n}_q)\longrightarrow C(S^{2n}_q),
\end{equation}
which is an analog of flattening an even dimensional sphere to an even dimensional ball of the same dimension.
For instance, in the case $n=1$, we have a graph homomorphism
\begin{equation}
\begin{tikzpicture}[auto,swap]
\tikzstyle{vertex}=[circle,fill=black,minimum size=3pt,inner sep=0pt]
\tikzstyle{edge}=[draw,->]
\tikzstyle{cycle1}=[draw,->,out=130, in=50, loop, distance=30pt]
   
\node[vertex,label=below:$v$] (0) at (0,0) {};
\node[vertex,label=left:$w_1$] (1) at (-0.5,-0.5) {};
\node[vertex,label=right:$w_2$] (2) at (0.5,-0.5) {};

\path (0) edge[cycle1] node[above] {$u$} (0);
\path (0) edge[edge] node[above left] {$f_1$} (1);
\path (0) edge[edge] node[above right] {$f_2$} (2);

\end{tikzpicture}\quad 
\begin{tikzpicture} 
\node (0) at (0,0) {};
\node (2) at (1.5,0) {};
\node (1) at (0,-0.5) {};
\path (0) edge[draw,->] node {} (2);
\end{tikzpicture}\quad
\begin{tikzpicture}[auto,swap]
\tikzstyle{vertex}=[circle,fill=black,minimum size=3pt,inner sep=0pt]
\tikzstyle{edge}=[draw,->]
\tikzstyle{cycle1}=[draw,->,out=130, in=50, loop, distance=30pt]
   
\node[vertex,label=right:$a$] (0) at (0,0) {};
\node[vertex,label=right:$b$] (1) at (0,-0.5) {};

\path (0) edge[cycle1] node[above] {$e$} (0);
\path (0) edge[edge] node[left] {$x$} (1);

\end{tikzpicture},
\end{equation}
where $v\mapsto a$, $w_1,w_2\mapsto b$, $u\mapsto e$, and $f_1, f_2\mapsto x$. Much in the same way, in the case $n=3$, 
we have the following graph homomorphism:
\begin{equation}
\begin{gathered}
\centering
\xymatrix{
\begin{tikzpicture}[scale=0.4,auto,swap]
\centering
\tikzstyle{vertex}=[circle,fill=black,minimum size=3pt,inner sep=0pt]
\tikzstyle{edge}=[draw,->]
\tikzset{every loop/.style={min distance=20mm,in=130,out=50,looseness=40}}
    \node[vertex] (1) at (-2,0) {};
    \node[vertex] (2) at (0,0) {};
    \node[vertex] (3) at (2,0) {};
    \node[vertex] (4) at (-2,-2) {};
     \node[vertex] (5) at (2,-2) {};
    \path (1) edge [draw, <-, anchor=center, loop above] node {} (1);
    \path (2) edge [draw, <-, anchor=center, loop above] node {} (2);
    \path (3) edge [draw, <-, anchor=center, loop above] node {} (3);
    \path (1) edge [edge] node {} (2);
    \path (2) edge [edge] node {} (3);
    \path (1) edge [edge] node {} (4);
    \path (3) edge [edge] node {} (5);
    \path (1) edge [edge] node {} (5);
    \path (2) edge [edge] node {} (4);
    \path (2) edge [edge] node {} (5);
    \path (3) edge [edge] node {} (4);
    \path (1) edge [edge,bend right] node {} (3);
\end{tikzpicture}\quad
\begin{tikzpicture} 
\node (0) at (0,0) {};
\node (2) at (1.5,0) {};
\node (1) at (0,-0.5) {};
\path (0) edge[draw,->] node {} (2);
\end{tikzpicture}\quad
\begin{tikzpicture}[scale=0.4,auto,swap]
\centering
\tikzstyle{vertex}=[circle,fill=black,minimum size=3pt,inner sep=0pt]
\tikzstyle{edge}=[draw,->]
\tikzset{every loop/.style={min distance=20mm,in=130,out=50,looseness=40}}
    \node[vertex] (1) at (-2,0) {};
    \node[vertex] (2) at (0,0) {};
    \node[vertex] (3) at (2,0) {};
    \node[vertex] (4) at (0,-2) {};
    \path (1) edge [draw, <-, anchor=center, loop above] node {} (1);
    \path (2) edge [draw, <-, anchor=center, loop above] node {} (2);
    \path (3) edge [draw, <-, anchor=center, loop above] node {} (3);
    \path (1) edge [edge] node {} (2);
    \path (2) edge [edge] node {} (3);
    \path (1) edge [edge] node {} (4);
    \path (2) edge [edge] node {} (4);
    \path (3) edge [edge] node {} (4);
    \path (1) edge [edge,bend right] node {} (3);
\end{tikzpicture}\quad.
}
\end{gathered}
\end{equation}
}
\end{example}

Next, if $F$ is an admissible subgraph of $E$, then we have the injective graph homomorphism
\begin{equation}\label{admlemfor}
f:F\underset{G}{\amalg}F\longrightarrow F\underset{G}{\amalg}E
\end{equation}
given by the universal property of the pushout applied to the obvious maps.
Similarly, from the universal property of the pushout applied to the obvious maps, 
we obtain the graph homomorphism
\begin{equation}\label{ig}
\iota_g:F\underset{G}{\amalg}E\longrightarrow E.
\end{equation}

\begin{proposition}\label{admlem}
The injective graph homomorphism (\ref{admlemfor}) is admissible. 
\end{proposition}
\begin{proof}
We have the following commutative diagram of graph homomorphisms:
\begin{equation}
\xymatrix{
F \ar[r] &
F\underset{G}{\amalg}F \ar[r] &
F\underset{G}{\amalg}E
\\
G \ar[u] \ar[r] & F \ar[r] \ar[u] & E \ar[u]
}\,.
\end{equation}
It is clear that both the left square and the outer rectangle are pushout diagrams in the category~{\rm OG}. It follows by standard categorical arguments 
that the right square is a pushout diagram (see~\cite[Proposition~11.10]{ahs90} for the dual result for pullbacks). The admissibility of $G\hookrightarrow F$ 
implies the admissibility of $F\hookrightarrow F\amalg_G F$ by Lemma~\ref{admpush}. We use Lemma~\ref{admpush} again, to infer the admissibility of 
$F{\amalg}_GF\hookrightarrow F{\amalg}_GE$ from the admissibility of $F\hookrightarrow E$ and $F\hookrightarrow F\amalg_G F$.
\end{proof}

\subsection{Line graphs}
 Now we explore a class of non-injective admissible graph homomorphisms whose induced $*$-homomorphisms are surjective.
 For starters, let us recall the notion of a~line graph~\cite{hn-60}. 
The {\em line graph} $LE=(LE^0,LE^1,s_{LE},t_{LE})$ of a graph $E$ is defined as follows:
\begin{equation}
LE^0:=E^1,\quad LE^1:=FP_2(E),\quad s_{LE}(ee')=e,\quad t_{LE}(ee')=e',\qquad e,e'\in E^1.
\end{equation} 
Next, consider the graph homomorphism
\begin{equation}\label{dualgraphmap}
f^0:LE^0\longrightarrow E^0,\quad f^0(e):=s_E(e),\quad f^1:LE^1\longrightarrow E^1,\quad f^1(ee'):=e.
\end{equation}
The following result is the discrete-topology version of~\cite[Proposition~2.6]{t-k21}: 
\begin{proposition}
Let $E$ be a row-finite graph without sinks. Then the graph homomorphism \eqref{dualgraphmap} is admissible and surjective. 
\end{proposition}
\begin{proof}
Since $E$ is row-finite, $f$ is automatically proper. 
Next, as for every $e\in E^1$ the map
\begin{equation}
(f^1)^{-1}(e)\ni ee'\longmapsto t_{LE}(ee')=e' \in (f^0)^{-1}(t_E(e))
\end{equation}
is clearly bijective, we conclude that $f$ satisfies the target-bijectivity condition. Furthermore, $f$ is automatically regular 
because all vertices in $E$ and $LE$ are regular. Finally, the surjectivity of $f$ follows from the assumption that there are no sinks.
\end{proof}
\noindent By Corollary~\ref{c*injsur}, $f:LE\to E$ induces an injective $*$-homomorphism $f^*_{C^*}:C^*(E)\to C^*(LE)$. 
However, it is known that this $*$-homomorphism is also surjective (e.g., see~\cite[Corollary~2.6]{i-r05}). Indeed, the inverse $*$-
homomorphism $(f^*_{C^*})^{-1}:C^*(LE)\longrightarrow C^*(E)$ is given by
\begin{equation}
P_e\longmapsto S_eS_e^*,\qquad S_{ee'}\longmapsto S_e S_{e'}S_{e'}^*.
\end{equation}
Thus, we obtain a class of examples of non-injective admissible graph homomorphisms
inducing surjective $*$-homomorphisms of graph C*-algebras.

\subsection{Locally derived graphs}
We end the section by discussing another class of non-injective admissible graph homomorphisms coming from finite group actions. 
A {\em base  graph} (or a {\em  voltage graph})~\cite[\S 2.1]{gt-87} is a graph $(E^0,E^1,s_E,t_E)$ along with a function 
$L:E^1\to\Gamma$, where $\Gamma$ is a  group. Given a base graph, one can construct a {\em derived graph} 
$(E^0_L, E^1_L,s_L, t_L)$~\cite[\S 2.1.1]{gt-87} (or a {\em skew-product graph}~\cite[Definition 2.1]{kp-99}) as follows:
\begin{gather}
E_L^0:=E^0\times\Gamma,\qquad E_L^1:=E^1\times\Gamma,
\nonumber\\
s_L((e,g)):=(s_E(e),g),\quad t_L((e,g)):=(t_E(e),L(e)g),\quad e\in E^1,\quad g\in\Gamma.
\end{gather}
There is a natural surjective graph homomorphism 
$\pi:E_L\to E$, called the {\em covering projection}, given by
\begin{equation}\label{skewproj}
\pi^0((v,g)):=v,\quad \pi^1((e,g)):=e,\quad v\in E^0,\quad e\in E^1,\quad g\in\Gamma.
\end{equation}
\begin{proposition}\label{covadm}
If $\,\Gamma$ is finite, then the covering projection $\pi:E_L\to E$ is admissible.
\end{proposition}
\begin{proof}
Since $\Gamma$ is finite, $\pi$ is proper. Next, for any $e\in E^1$, the map
\begin{equation}
(\pi^1)^{-1}(e)\ni (e,g)\longmapsto (t_E(e),L(e)g)\in (\pi^0)^{-1}(t_E(e))
\end{equation}
has an inverse given by $(t_E(e),h)\mapsto (e,L(e)^{-1}h)$, so $\pi$ satisfies the target-bijectivity condition. 
Finally, let $v\in E^0$ be a regular vertex and $g\in\Gamma$. Then, by
 the regularity of $v$,
there is $e\in E^1$ such that $s_L((e,g))=(v,g)$, so $(v,g)$ not a sink. Also, since $v$ is not an infinite emitter, neither is~$(v,g)$. 
We conclude thus that $(v,g)$ is a regular vertex for any $g\in\Gamma$. Hence, $\pi$ is regular, so we infer  that
$\pi$ is admissible, as claimed.
\end{proof}

It follows from Corollary~\ref{corgc*} that $\pi^*_{C^*}\colon C^*(E)\to C^*(E_L)$ is determined by the assignments:
\begin{equation}
P_v\longmapsto \sum_{g\in\Gamma}P_{(v,g)},\qquad S_e\longmapsto \sum_{g\in\Gamma}S_{(e,g)}\,.
\end{equation}

\begin{example}\label{shrinking}
{\rm Consider the surjective admissible graph homomorphism $\pi:A_n\to A_1$ given by collapsing all edges to one edge and all vertices to one vertex:
\begin{equation}
\begin{tikzpicture}[scale=0.4,auto,swap]
\centering
\tikzstyle{vertex}=[circle,fill=black,minimum size=3pt,inner sep=0pt]
\tikzstyle{edge}=[draw,->]
\node (0) at (5,0) {};
    \node[vertex] (1) at (0,0) {};
    \node[vertex] (2) at (1,-1) {};
    \node[vertex] (3) at (2,-1) {};
    \node[vertex] (4) at (3,0) {};
    \node[vertex] (5) at (3,1) {};
    \node[vertex] (6) at (2,2) {};
    \node[vertex] (7) at (1,2) {};
    \node[rotate=60] (8) at (0.5,1) {{\tiny $\cdots$}};
    \node (9) at (-2,1) {{\tiny $n$ edges}};
    \path (1) edge[edge] node {} (2);
    \path (2) edge[edge] node {} (3);
    \path (3) edge[edge] node {} (4);
    \path (4) edge[edge] node {} (5);
    \path (5) edge[edge] node {} (6);
    \path (6) edge[edge] node {} (7);
\end{tikzpicture}
\begin{tikzpicture}[scale=0.4,auto,swap]
\centering
\node (1) at (0,0) {};
\node (2) at (0,-1) {$\longrightarrow$};
\node (3) at (0,-2) {};
\node (4) at (0,-3) {};
\end{tikzpicture}
\quad
\begin{tikzpicture}[scale=0.4,auto,swap]
\centering
\tikzstyle{vertex}=[circle,fill=black,minimum size=3pt,inner sep=0pt]
\tikzstyle{edge}=[draw,->]
\node (0) at (0,0) {};
    \node[vertex] (1) at (0,-1) {};
    \node (3) at (0,-2) {};
    \path (1) edge [edge, anchor=center, loop above, min distance=20mm, in=130, out=50, looseness=40] node {} (1);
\end{tikzpicture}.
\end{equation}
Note that $\pi$ is the projection from a derived graph to its base graph. Indeed, let $\Gamma:=\mathbb{Z}/n\mathbb{Z}$ be the cyclic group of order $n$, 
and let $L:A^1_1\to\Gamma$ map the single edge in $A_1$ to $[1]\in \mathbb{Z}/n\mathbb{Z}$. 
Then it is clear that $A_n$ is isomorphic to $(A_1)_L$ and that $\pi$ is the covering projection.
The morphism $\pi$ induces an injective $*$-homomorphism $\pi^*_{C^*}:C^*(A_1)\to C^*(A_n)$,
 which combined with the standard identification $C^*(A_n)\cong C(S^1)\otimes M_n(\mathbb{C})$ (e.g., see~\cite[Example~2.14]{i-r05}), yields
\begin{equation}\label{nsqrt}
C(S^1)\longrightarrow C(S^1)\otimes M_n(\mathbb{C}),\qquad u\longmapsto \sum_{i=1}^{n-1}(1\otimes E_{i(i+1)})+u\otimes E_{n1}.
\end{equation}
Here $u$ is the unitary generator of $C(S^1)$ and $E_{ij}$, $i,j=1,\ldots, n$, are the matrix units of $M_n(\mathbb{C})$. 
To end with, observe that precomposing the map \eqref{nsqrt} with $C(S^1)\ni u\mapsto u^n\in C(S^1)$ produces the standard tensorial inclusion
\begin{equation}
C(S^1)\longrightarrow C(S^1)\otimes M_n(\mathbb{C}),\qquad u\longmapsto u\otimes 1.
\end{equation}}
\end{example}

%%%%%%%%%%%%%%%%%%%%%%%%%%%%%%%%%%%%%%%%%%%%%%%%%%%%%%%%%%%%%%%%%%%%%%%

Now, let us generalize the construction of a derived graph. For a graph $E$, let $\{E_i\}_{i\in I}$ be a family of 
pairwise-disjoint subgraphs of $E$,
 and let $\{\Gamma_i\}_{i\in I}$ be a family of groups. Assume that there is a labelling map 
$L_i:t_E^{-1}(E^0_i)\to\Gamma_i$ for every $i\in I$,
  and combine them to a map 
  \begin{equation}
 \mathcal{E}^1:=\bigcup_{i\in I}t_E^{-1}(E^0_i)\stackrel{ \mathcal{L}}{\longrightarrow} \bigsqcup_{i\in I}\Gamma_i=:
\mathcal{G}.
  \end{equation} 
(Note that we can view $\mathcal{G}$ as a groupoid in the obvious way.)
  We call the pair $(E,\mathcal{L})$ a {\em base graph}. The idea of constructing 
a locally derived graph is that in the base graph we replace 
  every subgraph $E_i$ by its derived graph, and unfold (keeping the source fixed) 
every edge that does not belong to any of the subgraphs but
ends in a subgraph. More precisely, we have the following:
 \begin{definition}
 The {\em locally derived graph} $(E_\mathcal{L}^0,E^1_\mathcal{L},s_\mathcal{L},t_\mathcal{L})$ of a base graph $(E,
\mathcal{L})$ is given by:
 \begin{gather}
 E_\mathcal{L}^0:= \Big(E^0\setminus  \bigcup_{i\in I} E^0_i\Big)\; \sqcup\, \;\bigcup_{i\in I} (E^0_i\times\Gamma_i),\quad
E_\mathcal{L}^1:= \Big(E^1\setminus  \bigcup_{i\in I} t_E^{-1}(E^0_i)\Big)\; \sqcup\, \;\bigcup_{i\in I} 
(t_E^{-1}(E^0_i)\times\Gamma_i),
\nonumber \\
  \begin{cases}
 s_\mathcal{L}((e,g)):=
 (s_E(e),g)\quad\text{for}\quad e\in E^1_i,\,g\in\Gamma_i\,,i\in I,
 \\
s_\mathcal{L}((e,g)):=
 (s_E(e),1_j)\quad\text{for}\quad e\in  t_E^{-1}(E^0_i)\setminus E^1_i,\,g\in\Gamma_i\,,\, s_E(e)\in E^0_j\,,\,i,j\in I,
 \\
s_\mathcal{L}((e,g)):=
 s_E(e)\quad\text{for}\quad e\in  t_E^{-1}(E^0_i)\setminus E^1_i,\,g\in\Gamma_i\,,\, \,i\in I,\,
s_E(e)\not\in \bigcup_{j\in I}E^0_j\,,
 \\
  s_\mathcal{L}(e):=(s_E(e),1_i)\quad\text{for}\quad e\not\in\mathcal{E}^1,\, s_E(e)\in E^0_i,\, i\in I,
  \\
  s_\mathcal{L}(e):=s_E(e)\quad\text{for}\quad e\not\in\mathcal{E}^1,\, s_E(e)\not\in \bigcup_{i\in I}E^0_i\, ,
 \end{cases} \nonumber\\
 \begin{cases}
 t_\mathcal{L}((e,g)):=
 (t_E(e),\mathcal{L}(e)g)\quad\text{for}\quad e\in t^{-1}_E(E^0_i),\,g\in\Gamma_i\,,\, i\in I,
 \\
  t_\mathcal{L}(e):=t_E(e)\quad\text{for}\quad e\not\in\mathcal{E}^1.
 \end{cases} \label{targetL}
 \end{gather}
 Here $1_i$\,, $i\in I$, is the neutral element of $\,\Gamma_i$.
 \end{definition}
 
 \begin{lemma}
 The following assignments
 \[
 \begin{cases}
 \pi^0_E((v,g)):=v\quad\text{for}\quad v\in E^0_i\,, g\in\Gamma_i\,, i\in I,\\
 \pi^0_E(v):=v\quad\text{for}\quad v\notin \bigcup_{i\in I}E_i^0\,,\\
 \pi^1_E((e,g)):=e\quad\text{for}\quad e\in t_E^{-1}(E^0_i)\,, g\in\Gamma_i\,, i\in I,\\
 \pi^1_E(e):=e\quad\text{for}\quad e\notin\mathcal{E}^1,
 \end{cases}
 \]
 define a surjective graph homomorphism $\pi_E:E_\mathcal{L}\to E$.
 \end{lemma}
 \begin{proof}
 First, note that, for all $e\in E^1_i$, $g\in\Gamma_i\,, i\in I$, we have
 \begin{equation}\label{trivial1}
 \pi^0_E(s_\mathcal{L}((e,g)))=\pi^0_E((s_E(e),g))=s_E(e)=s_E(\pi^1_E((e,g))).
\end{equation}
Much in the same way, if $e\in t_E^{-1}(E^0_i)\setminus E^1_i$, $g\in\Gamma_i$, and $s_E(e)\in E^0_j$ for some $i,j\in I$, then
\begin{equation}
 \pi^0_E(s_\mathcal{L}((e,g)))=\pi^0_E((s_E(e),1_j))=s_E(e)=s_E(\pi^1_E((e,g))).
\end{equation}
Next, if $e\in  t_E^{-1}(E^0_i)\setminus E^1_i$, $s_E(e)\not\in \bigcup_{j\in I}E^0_j$, and $g\in\Gamma_i$ for some $i\in I$, then
\begin{equation}
\pi^0_E(s_\mathcal{L}((e,g)))=\pi^0_E(s_E(e))=s_E(e)=s_E(\pi^1_E((e,g))).
\end{equation}
Now, if $e\not\in\mathcal{E}^1$ and $s_E(e)\in E^0_i$ for some $i\in I$, we obtain 
\begin{equation}
\pi^0_E(s_\mathcal{L}(e))=\pi^0_E((s_E(e),1_i))=s_E(e)=s_E(\pi^1_E(e)).
\end{equation}
Finally, if $e\not\in\mathcal{E}^1$ and $s_E(e)\not\in \bigcup_{i\in I}E^0_i$, 
then 
\begin{equation}\label{trivial2}
\pi^0_E(s_\mathcal{L}(e))=\pi^0_E(s_E(e))=s_E(e)=s_E(\pi^1_E(e)).
\end{equation}
The calculations for the target map are much simpler and analogous to \eqref{trivial1} and~\eqref{trivial2}.
 \end{proof}
We call the graph homomorphism $\pi_E:E_\mathcal{L}\to E$ a \emph{projection folding}. Under additional assumptions, we can prove that 
this graph homomorphism is admissible:
 \begin{proposition}\label{derprop}
If, for all $i\in I$, the group $\Gamma_i$ is finite and the inclusion $E_i\hookrightarrow E$ is regular, then the   projection 
folding $\pi_E\colon E_\mathcal{L}\to E$ is admissible.
 \end{proposition}
 \begin{proof}
First, since all $\Gamma_i$ are finite, $\pi_E$ is proper. 
To prove target bijectivity, we consider two cases. If $e\in t_E^{-1}(E_i^0)$ for some $i\in I$, then we obtain a map
\begin{equation}
(\pi^1_E)^{-1}(e)\ni (e,g)\longmapsto t_\mathcal{L}((e,g))=(t_E(e),\mathcal{L}(e)g)\in (\pi^0_E)^{-1}(t_E(e))
\end{equation}
whose inverse is given by $(t_E(e),g)\mapsto (e,\mathcal{L}(e)^{-1}g)$. Next, if $e\notin\mathcal{E}^1$, then the map
\begin{equation}
(\pi^1_E)^{-1}(e)=\{e\}\longrightarrow \{t_\mathcal{L}(e)\}=\{t_E(e)\}=(\pi^0_E)^{-1}(t_E(e))
\end{equation}
is, clearly, a bijection.
Furthermore, the regularity of $\pi_E$ at $v\in E^0_i$, $i\in I$, follows from regularity of the inclusions $E_i\hookrightarrow E$ and the reasoning as in 
the proof of Proposition~\ref{covadm}. Finally, the regularity at other vertices $v\in E^0\setminus\bigcup_{i\in I}E^0_i$ is immediate.
\end{proof}

\begin{example}
{\rm 
Consider the graph $R_2$ of the Cuntz algebra $\mathcal{O}_2$, i.e.\ $R^0_2:=\{v\}$, $R^1_2:=\{e,f\}$.
Take the subgraph $E$ of $R_2$ given by $E^0:=\{v\}$ and $E^1:=\{e\}$, and consider
\[
\mathcal{L}:t_{R_2}^{-1}(v)=R^1_2\longrightarrow \mathbb{Z}/2\mathbb{Z},\quad e\longmapsto \gamma,\quad f\longmapsto 1\,,
\]
where $\gamma$ is the generator of $\mathbb{Z}/2\mathbb{Z}$. The graph $(R_2)_\mathcal{L}$ is presented in the picture below.
\[
\begin{tikzpicture}[scale=0.8,auto,swap]
\centering
\tikzstyle{vertex}=[circle,fill=black,minimum size=3pt,inner sep=0pt]
\tikzstyle{edge}=[draw,->]
\tikzset{every loop/.style={min distance=20mm,in=130,out=50,looseness=40}}
    \node[vertex] (1) at (-2,0) {};
    \node[vertex] (2) at (2,0) {};
    \node (3) at (-2.2,-0.5) {\tiny $(v,1)$};
    \node (4) at (2.2,-0.5) {\tiny $(v,\gamma)$};
    \node (5) at (-2,1.5) {\tiny $(f,1)$};
    \node (6) at (0,1) {\tiny $(e,1)$};
    \node (7) at (0,0.275) {\tiny $(f,\gamma)$};
    \node (8) at (0,-1) {\tiny $(e,\gamma)$};
    \path (1) edge [edge, anchor=center, loop above] node {} (1);
    \path (1) edge [edge, bend left] node {} (2);
    \path (2) edge [edge, bend left] node {} (1);
    \path (1) edge [edge] node {} (2);
\end{tikzpicture}
\]
It is clear that $C^*((R_2)_\mathcal{L})$ is simple (e.g., see~\cite[Proposition~4.2]{i-r05}). Next, by Corollary~\ref{corgc*}, the admissible graph 
homomorphism $\pi_{R_2}:(R_2)_\mathcal{L}\to R_2$ induces a unital $*$-homomorphism 
$(\pi_{R_2})^*_{C^*}\colon\mathcal{O}_2\to C^*((R_2)_\mathcal{L})$ given on generators by
\[
S_e\longmapsto S_{(e,1)}+S_{(e,\gamma)},\qquad S_f\longmapsto S_{(f,1)}+S_{(f,\gamma)}.
\]
The above $*$-homomorphism is injective because $\mathcal{O}_2$ is simple. However,
since $K_0(C^*((R_2)_\mathcal{L}))$ \mbox{$=\mathbb{Z}/2\mathbb{Z}$} by~\cite[Theorem~7.16]{i-r05}, 
and $K_0(\mathcal{O}_2)=0$, it cannot be surjective. Hence, $\mathcal{O}_2$ is a proper subalgebra of 
$C^*((R_2)_\mathcal{L})$ and the $K_0$-class of $1\in  C^*((R_2)_\mathcal{L}))$ is zero. On the other hand, it follows 
from~\cite[Theorem~6.5]{m-r95} that $C^*((R_2)_\mathcal{L})$ is stably isomorphic with~$\mathcal{O}_3$. Better still, as the $K_0$-class of $1\in M_2(\mathcal{O}_3)$ is zero, using again \cite[Theorem~6.5]{m-r95}, one can conclude
 that $C^*((R_2)_\mathcal{L})\cong M_2(\mathcal{O}_3)$.\footnote{
We owe this argument to Jack Spielberg.}
}
\end{example}

Let $F$ be a subgraph of a graph $E$ such that $t_E^{-1}(F^0)\subseteq F^1$ and let $(F,\mathcal{L})$ be a base graph. 
Then any family $\{F_i\}_{i\in I}$ of pairwise-disjoint subgraphs of $F$ is a family of 
pairwise-disjoint subgraphs of $E$ and $t^{-1}_E(F^0_i)=t_F^{-1}(F^0_i)$ for all $i\in I$.
Therefore, given $\mathcal{L}\colon\mathcal{F}^1\to\mathcal{G}$, we can construct both 
locally derived graphs $F_\mathcal{L}$ and $E_\mathcal{L}$.
\begin{proposition}\label{adml}
If $F$ is an admissible subgraph of a graph $E$ and $(F,\mathcal{L})$ is a base graph, then $F_\mathcal{L}$ is 
an admissible subgraph of $E_{\mathcal{L}}$.
\end{proposition}
\begin{proof}
We need check the conditions (A1) and (A2) of Definition~\ref{admissiblesubgraph}. 
To prove (A1), suppose that $w\in F^0_\mathcal{L}\cap {\rm reg}(E_\mathcal{L})$ is such that 
$t_{E_\mathcal{L}}(s_{E_\mathcal{L}}^{-1}(w))\subseteq E^0_{\mathcal{L}}\setminus F^0_{\mathcal{L}}$. 
First, note that, by the construction of $E_\mathcal{L}$, $w$ cannot be of the form $(v,g)\in F^0_i\times\Gamma_i$, where $g\neq 1_i$, $i\in I$. 
Therefore, $w$ has to be of the form $(v,1_i)\in F^0_i\times\Gamma_i$ for some $i\in I$, or $v\in F^0\setminus \bigcup_{i\in I}F^0_i$. 
In both cases, it follows immediately from~\eqref{targetL} that the condition 
$t_{E_\mathcal{L}}(s_{E_\mathcal{L}}^{-1}(w))\subseteq E^0_{\mathcal{L}}\setminus F^0_{\mathcal{L}}$ implies that 
$t_{E}(s_{E}^{-1}(v))\subseteq E^0\setminus F^0$, which contradicts the fact that $E^0\setminus F^0$ is saturated. 
Hence, $E^0_\mathcal{L}\setminus F^0_\mathcal{L}$ is saturated. 
Finally, using again~\eqref{targetL} combined with the condition (A2) for $F\subseteq E$, we conclude the condition (A2) for 
$F_\mathcal{L}\subseteq E_\mathcal{L}$.
\end{proof}

\section{Pullbacks of graph algebras from pushouts of graphs}
\noindent
In this section, we prove our main theorems stating under which conditions the two contravariant functors given by 
Lemma~\ref{contrafp} and Theorem~\ref{contralthm} turn
pushouts of directed graphs into pullbacks of algebras. 
\subsection{Path algebras}
We begin with  the pushout-to-pullback theorem for path algebras.
\begin{theorem}\label{pushpullpaththm}
Let the diagram
\begin{equation}\label{pushdiag0}
\xymatrix{
&
E\underset{G}{\amalg}F
&\\
E
\ar[ur]^{(\iota^0_{E},\iota^1_{E})}& & 
F
\ar[ul]_{(\iota^0_{F},\iota^1_{F})}\\
&
G
\ar[ur]_{(g^0,g^1)}\ar[ul]^{(f^0,f^1)}&
}
\end{equation}
be a pushout diagram in the category of graphs and proper graph homomorphisms such that
\vspace*{-2mm}
\begin{enumerate}
\item
both $f^0$ and $g^0$ are injective (vertex injectivity),
\item
$t_\amalg(x)=s_\amalg(y)\;\Rightarrow\;(x,y\in\iota_{E^1}(E^1)\text{ or }x,y\in\iota_{F^1}(F^1))$
(one color),
\item
$f^1$ or $g^1$ is injective (one-sided injectivity).
\end{enumerate}
Then, for any field $k$, the contravariant functor $\mathrm{Map}_f(\cdot,k)$ 
transforms the above pushout diagram to the following one-surjective pullback
diagram in the category of  algebras over~$k$:
\begin{equation*}
\xymatrix{
&
k\left(E\underset{G}{\amalg}F\right)
\ar[dl]_{\iota_E^*} \ar[dr]^{\iota_F^*}&\\
kE \ar[dr]_{f^*} & & 
kF\ar[dl]^{g^*}\\
&
kG\,.&
}
\end{equation*}
Here $f$ and $g$ are the maps induced by the graph homomorphisms
$E\stackrel{(f^0,f^1)}{\longleftarrow} G\stackrel{(g^0,g^1)}{\longrightarrow} F$, respectively.
Furthermore, if $E^0$ and $F^0$ are finite, then all algebras in the above diagram and homomorphisms 
between them are unital, and the diagram is a pullback diagram in the category of unital algebras.
\end{theorem}
\begin{proof}
To begin with,  Lemma~\ref{oneinjective} guarantees that \eqref{pushdiag0} is a pushout diagram in the 
category of graphs and proper graph homomorphisms for any proper graph homomorphisms
\begin{equation}
E\stackrel{(f^0,f^1)}{\longleftarrow} G\stackrel{(g^0,g^1)}{\longrightarrow} F
\end{equation}
 satisfying the assumptions
(1) through~(3). Next, as proper graph homomorphisms induce finite-to-one maps between the path spaces,
and injective graph homomorphisms induce injective maps between the path spaces,
 Lemma~\ref{answer} and Lemma~\ref{oneinjective} yield the following pushout diagram in the category of sets
and finite-to-one maps:
\begin{equation}
\xymatrix{
&
FP(E)\underset{FP(G)}{\amalg}FP(F)
&\\
FP(E)
\ar[ur]^{\iota_{E}}& & 
FP(F).
\ar[ul]_{\iota_{F}}\\
&
FP(G)
\ar[ur]_{g}\ar[ul]^{f}&
}
\end{equation}
Now, Lemma~\ref{contrafp}
 turns this diagram into the  commutative diagram in the category of algebras:
\begin{equation}
\xymatrix{
&
k\left(E\underset{G}{\amalg}F\right)
\ar[dl]_{\iota_E^*} \ar[dr]^{\iota_F^*}&\\
kE \ar[dr]_{f^*} & & 
kF\ar[dl]^{g^*}\\
&
kG\,.&}
\end{equation}
Furthermore, setting $K=k$ in Lemma~\ref{finitepushtopull}, we conclude that the above diagram is
a pullback diagram in the category of sets and maps.
We can combine these two facts to conclude that the above diagram is 
a pullback diagram in the category of algebras.

To prove the last part of the theorem, assume that both $E^0$ and $F^0$ are finite.
As graph homomorphisms are assumed to be proper, not only the finiteness of 
$E^0\amalg_{G^0}F^0$, but also the finiteness of $G^0$ follow from the finiteness of $E^0$ and~$F^0$,
so all algebras are unital, as claimed. Next, the unitality of all homomorphisms follows from
Lemma~\ref{contrafp}. Finally, under these circumstances, the maps witnessing the universality 
in the category of unital algebras are
evidently unital, so the above diagram is 
a pullback diagram in the category of unital algebras.
\end{proof}

\subsection{Leavitt path algebras and graph C*-algebras}

For starters, if $f\colon G\to E$ is an injective admissible graph homomorphism, then it follows from Remark~\ref{remark3.5} and \cite[Corollary~2.4.13(ii)]{aasm17}
that  $\ker f_L^*$ is the ideal generated by
\begin{equation}
\left\{[\chi_v]~|~v\in E^0\setminus f^0(G^0)\right\}\cup \Big{\{}[\chi_w]-\!\!\!\!\!\!\!\!\!\!\!\!\!\!\!\!\!
\sum_{e\in s_E^{-1}(w)\cap t_E^{-1}(f^0(G^0))}\!\!\!\!\!\!\!\!\!\!\!\!\!\!\!\!\![\chi_e][\chi_{e^*}]~|
~w\in B_{E^0\setminus f^0(G^0)}\Big{\}}.
\end{equation}
As $E^0\setminus f^0(G^0)$ is hereditary by Proposition~\ref{hermorph}, we can combine the above with  \cite[Lemma~2.4.6]{aasm17} to obtain:
\begin{align}\label{kerspan}
\ker f^*_L&={\rm span}\Big{\{}[\chi_\alpha][\chi_v][\chi_{\beta^*}]~\Big{|}~\substack{\alpha,\beta\in FP(E),~v\in E^0\setminus 
f^0(G^0),\\t_E(\alpha)=t_E(\beta)=v}\Big{\}}\nonumber\\
&+{\rm span}\Big{\{}[\chi_\mu]\Big{(}[\chi_w]-\hspace*{-12mm}\sum_{e\in s^{-1}_E(w)\cap\, t^{-1}_E(f^0(F^0))}
\hspace*{-12mm}[\chi_e][\chi_{e^*}]\Big{)}
[\chi_{\nu^*}]~\Big{|}~\substack{\mu,\nu\in FP(E),~w\in B_{E^0\setminus f(G^0)},\\t_E(\mu)=t_E(\nu)=w}\Big{\}}.
\end{align}

Before we consider our pushout-to-pullback theorem for Leavitt path algebras, we need further technical results.  
\begin{lemma}\label{kerver}
Let $G$ and $E$ be arbitrary graphs, and let $(f^0,f^1)\colon G\to E$ be an admissible graph homomorphism.
Then
\begin{equation*}
\forall v\in E^0\colon\quad [\chi_v]\in\ker f^*_L\quad\iff\quad v\in E^0\setminus f^0(G^0).
\end{equation*}
\end{lemma}
\begin{proof}
The claim of the lemma is equivalent to the following statement
\begin{equation}
\forall v\in E^0\colon\quad \sum_{w\in(f^0)^{-1}(v)}[\chi_w]=0\quad\iff\quad (f^0)^{-1}(v)=\emptyset.
\end{equation} 
The implication ($\Leftarrow$) is true due to the convention that the sum over the empty set equals~$0$. The other implication follows from the fact that the elements $[\chi_w]$, $w\in(f^0)^{-1}(v)$, are linearly 
independent by \cite[Corollary~1.5.12]{aasm17}.
\end{proof}

\begin{lemma}\label{kerint}
Let $(\iota^0_E,\iota^1_E):E\to P$ and $(\iota^0_F,\iota_F^1):F\to P$ be admissible graph homomorphisms, and let 
$(\iota_E)^*_L$ and $(\iota_F)^*_L$ be the induced $\mathbb{Z}$-graded algebra homomorphisms. Then
\begin{equation*}
\ker (\iota_E)^*_L\cap\ker (\iota_F)^*_L=\{0\}\quad\iff\quad P^0=\iota_E^0(E^0)\cup \iota_F^0(F^0).
\end{equation*}
\end{lemma}
\begin{proof}
Since $\ker(\iota_E)^*_L\cap\ker(\iota_F)^*_L$ is a $\mathbb{Z}$-graded ideal, by [1, Theorem 2.5.8], 
\[
\ker(\iota_E)^*_L\cap\ker(\iota_F)^*_L=\{0\}\iff \{v\in P^0~|~[\chi_v]\in \ker(\iota_E)^*_L\cap\ker(\iota_F)^*_L\}=\emptyset.
\]
Finally, by Lemma 6.2, the above set equals $P^0\setminus (\iota_E^0(E^0)\cup\iota_F^0(F^0))$, which ends the proof.
\end{proof}

We are now ready for our main result:
\begin{theorem}\label{main}
Let $(f^0,f^1)$ and $(g^0,g^1)$ be admissible graph homomorphisms and let
\begin{equation}\label{pushdiag}
\xymatrix{
&
E\underset{G}{\amalg}F
&\\
E
\ar[ur]^{(\iota_E^0,\iota_E^1)}& & 
F
\ar[ul]_{(\iota_F^0,\iota_F^1)}\\
&
G
\ar[ur]_{(g^0,g^1)}\ar[ul]^{(f^0,f^1)}&
}
\end{equation}
be a pushout diagram in the  category {\rm OG} of graphs and graph homomorphisms. Assume also that
\begin{enumerate}
\item[{\rm (P1)}] 
$(f^0,f^1)$ is injective,
\item[{\rm (P2)}]
 $g^0$ restricted to $(f^0)^{-1}(B_{E^0\setminus f^0(G^0)})$ is injective,
\item[{\rm (P3)}]
$\iota_E^0(B_{E^0\setminus f^0(G^0)})\subseteq B_{P^0\setminus\iota_F^0(F^0)}$.
\end{enumerate}
Then, for any field $k$, there exists the commutative diagram of the induced $\mathbb{Z}$-graded algebra 
\mbox{homo}\-mor\-phisms
\begin{equation}\label{pulldiag}
\xymatrix{
&
L_k(E\underset{G}{\amalg}F)
\ar[dl]_{(\iota_E)^*_L} \ar[dr]^{(\iota_F)^*_L}&\\
L_k(E)
\ar[dr]_{f^*_L}& & 
L_k(F)
\ar[dl]^{g^*_L}\\
& L_k(G) .&
}
\end{equation}
Moreover, it is a left-surjective pullback diagram in the category $\mathrm{ZKA}$ of $\mathbb{Z}$-graded
algebras and $\mathbb{Z}$-graded algebra homomorphisms.
Finally, if $E^0$ and $F^0$ are finite, then all algebras in the above 
diagram and homomorphisms between them are unital, and the diagram is a pullback diagram in the category {\rm ZUKA} 
of unital $\mathbb{Z}$-graded
algebras and unital $\mathbb{Z}$-graded algebra homomorphisms.
\end{theorem}
\begin{proof}
Throughout the proof, let $P:=E\amalg_GF$. 
The existence of the diagram \eqref{pulldiag} follows from Lemma~\ref{admpush} and Theorem~\ref{contralthm}.
Its commutativity is due to the commutativity of~\eqref{pushdiag}, the surjectivity of $f^*_L$ is due to (P1) and Corollary~\ref{injsurcor}.
To show that~\eqref{pulldiag} is a pullback,
 by~\cite[Proposition~3.1]{gk-p99} (cf.~\cite[Lemma~4.1]{hkt20}), we have to prove that
\begin{enumerate}
\item[(1)] $\ker(\iota_E)^*_L\cap\ker(\iota_F)^*_L=\{0\}$,
\item[(2)] $(g^*_L)^{-1}(f^*_L(L_k(E)))=(\iota_F)^*_L(L_k(P))$,
\item[(3)] $\ker f^*_L=(\iota_E)^*_L(\ker(\iota_F)^*_L)$.
\end{enumerate}
Since~\eqref{pushdiag} is a pushout, we infer that $P^0=\iota_E^0(E^0)\cup\iota_F^0(F^0)$.
Therefore, the condition (1) follows from~Lemma~\ref{kerint}. Furthermore, as (P1) implies the injectivity of 
$(\iota_F^0,\iota_F^1)$, we conclude that
both $f^*_L$ and $(\iota_F)^*_L$ are surjective, which proves the condition (2). Finally, as the diagram~\eqref{pulldiag} is 
commutative, to prove the condition (3), it suffices to show that $\ker f^*_L\subseteq (\iota_E)^*_L(\ker(\iota_F)^*_L)$. Note that 
$(\iota_E)^*_L(\ker(\iota_F)^*_L)$ is a vector subspace, so it is enough to prove that all the elements that span $\ker f^*_L$ 
by~\eqref{kerspan} belong to~$(\iota_E)^*_L(\ker(\iota_F)^*_L)$. 

First, let us prove that $(\iota_E)^*_L(\ker(\iota_F)^*_L)$ contains the generators of $\ker f^*_L$. 
For starters, since~\eqref{pushdiag} is a pushout, $\iota_E^0(E^0\setminus f^0(G^0))\subseteq P^0\setminus\iota_F^0(F^0)$ 
and $(\iota_E^0)^{-1}(\iota_E^0(v))=\{v\}$ for any $v\in E^0\setminus f^0(G^0)$.
In turn, this implies that $(\iota_E)^*_L([\chi_{\iota_E^0(v)}])=[\chi_v]$, 
where $[\chi_{\iota_E^0(v)}]\in\ker(\iota_F)^*_L$ by Lemma~\ref{kerver}. 

Now we have to take care of breaking vertices, i.e.\ to show that
\begin{equation}
\forall\, w\in B_{E^0\setminus f^0(G^0)}\colon\;\Omega_w:=[\chi_w]-\!\!\!\!\!\!\!\!\!\!\!\!\sum_{e\in s^{-1}_E(w)\cap f^1(G^1)}\!\!\!\!\!\!\!\!\!\!\!\!
[\chi_e][\chi_{e^*}]
\in (\iota_E)^*_L(\ker(\iota_F)^*_L).
\end{equation}
To this end, let us take $v:=(f^0)^{-1}(w)$ and note that $v\in {\rm reg}(G)$ by the injectivity of~$f^0$. 
Next, the commutativity of the pushout diagram and the condition (P3)  imply that 
 $\iota_F^0(g^0(v))=\iota_E^0(f^0(v))\in B_{P^0\setminus \iota_F^0(F^0)}$, so $g^0(v)$ cannot be an infinite emitter. Also, 
$g^0(v)$ cannot be a 
sink because $v$ is regular. Consequently, $g^0(v)\in {\rm reg}(F)$. Furthermore, note that $\iota_E^0(w)=\iota_E^0(w')$ implies that $w'\in f^0(G^0)$. 
Better still, 
\begin{equation}\label{star}
(f^0)^{-1}((\iota_E^0)^{-1}(\iota_E^0(w)))=(g^0)^{-1}((\iota_F^0)^{-1}(\iota_F^0(f^0(v))))=(g^0)^{-1}(g^0(v))\subseteq {\rm reg}(G)
\end{equation}
because $g$ is regular and $g^0(v)\in{\rm reg}(F)$. 

Now, we need to argue that $w$ is the only element of
\begin{equation}
V_w:=(\iota_E^0)^{-1}(\iota_E^0(w))
\end{equation}
that is an infinite emitter. Indeed, suppose that $|s_E^{-1}(w)|=\infty$ for $w'\in V_w$. Then, as $(f^0)^{-1}(w)\in {\rm reg}(G)$ by~\eqref{star}, we 
conclude that $w'\in {\rm reg}(f(G))$ by the injectivity of $f^0$. Hence, $w'\in B_{E^0\setminus f^0(G^0)}$, so $w'=w$ by the condition (P2).

We are now ready to define the following element of $L_k(P)$:
\begin{equation}
\widetilde{\Omega}_w:=
[\chi_{\iota_F^0(g^0(v))}]-\!\!\!\!\!\!\!\sum_{e\in s_F^{-1}(g^0(v))}\!\!\!\!\!\!\![\chi_{\iota_F^1(e)}][\chi_{\iota_F^1(e)^*}]-
\!\!\!\!\!\!\!\!\!\!\!\!\!\!\!\!\!\sum_{e'\in s_E^{-1}(V_w\setminus\{w\})\setminus 
f^1(G^1)}\!\!\!\!\!\!\!\!\!\!\!\!\!\!\!\!\![\chi_{\iota_E^1(e')}][\chi_{\iota_E^1(e')^*}].
\end{equation}
First, using the injectivity of $\iota_F$ and the fact that every $e'\notin\iota_F^1(F^1)$, we check that
\begin{equation}
(\iota_F)^*_L(\widetilde{\Omega}_w)=[\chi_{g^0(v)}]-\!\!\!\!\!\!\!\sum_{e\in s_F^{-1}(g^0(v))}\!\!\!\!\!\!\![\chi_e][\chi_{e^*}]=0,
\end{equation}
so $\widetilde{\Omega}_w\in\ker(\iota_F)^*_L$. 

Next, to prove that $(\iota_E)^*_L(\widetilde{\Omega}_w)=\Omega_w$, first we show that
\begin{equation}\label{doublestar}
(\iota_E^1)^{-1}(\iota_F^1(s_F^{-1}(g^0(v)))=s_E^{-1}(V_w)\cap f^1(G^1).
\end{equation}
If $x$ belongs to the left-hand side, then $x\in f^1(G^1)$. Next, as $\iota_E^1(x)=\iota_F^1(y)$ for $y\in s_F^{-1}(g^0(v))$, we can compute
\begin{equation}
\iota^0_E(s_E(x))=s_F(\iota_F^1(x))=
s_P(\iota_F^1(y))=\iota_F^0(s_F(y))=\iota_F^0(g^0(v))=\iota_E^0(f^0(v))=\iota_F^0(w).
\end{equation}
Hence, the left-hand side is contained in the right-hand side. To prove the opposite inclusion, take 
$f^1(z)\in s_E^{-1}(V_v)$. Then \begin{equation}\iota_F^0(s_E(f^1(z)))=\iota_E^0(w)=\iota_E^0(f^0(v))=\iota_F^0(g^0(v)).
\end{equation}
On the other hand,
\begin{equation}
\iota_E^0(s_E(f^1(z)))=\iota_E^0(f^0(s_G(z)))=\iota_F^0(g^0(s_G(z))).
\end{equation}
Combining these two calculations with the injectivity of $\iota_F^0$, we infer that $g^0(s_G(z))=g^0(v)$, so 
\begin{equation}
s_F(g^1(z))=g^0(s_G(z))=g^0(v).
\end{equation}
Hence,
\begin{equation}
\iota_E^1(f^1(z))=\iota_F^1(g^1(z))\in\iota_F^1(s_F^{-1}(g^0(v))),
\end{equation}
which proves the desired inclusion.

Now, using the injectivity of $\iota_E^1$ on $E^1\setminus f^1(G^1)$ and~\eqref{doublestar}, we can compute:
\begin{align}
(\iota_E)^*_L(\widetilde{\Omega}_w)&=(\iota_E)^*_L([\chi_{\iota_E^0}])-\!\!\!\!\!\!\!\!\!\!\!\!\!\!\sum_{e\in (\iota_E^1)^{-1}(\iota_F^1(s_F^{-1}(g^0(v))))}
\!\!\!\!\!\!\!\!\!\!\!\!\!\!
[\chi_e][\chi_{e^*}]-\!\!\!\!\!\!\!\sum_{e'\in s_E^{-1}(V_w\setminus\{w\})\setminus f^1(G^1)}\!\!\!\!\!\!\![\chi_{e'}][\chi_{e'^*}]
\nonumber \\
&=\sum_{w'\in V_w}[\chi_{w'}]-\!\!\!\!\!\!\sum_{e\in s_E^{-1}(V_w)\cap f^1(G^1)}\!\!\!\!\!\![\chi_e][\chi_{e^*}]-
\!\!\!\!\!\!\sum_{e'\in s_E^{-1}(V_w\setminus\{w\})\setminus f^1(G^1)}\!\!\!\!\!\!
[\chi_{e'}][\chi_{e'^*}]
\nonumber\\
&=[\chi_w]-\!\!\!\!\!\!\!\sum_{e\in s_E^{-1}(w)\cap f^1(G^1)}\!\!\!\!\!\!\![\chi_e][\chi_{e^*}]+
\!\!\!\!\sum_{w'\in V_w\setminus\{w\}}\!\!\!\![\chi_{w'}]
\nonumber\\
&\phantom{=[\chi_w].}-\!\!\!\!\!\!\sum_{e\in s_E^{-1}(V_w\setminus\{w\})\cap f^1(G^1)}\!\!\!\!\!\![\chi_e][\chi_{e^*}]
-\!\!\!\!\!\sum_{e'\in s_E^{-1}(V_w\setminus\{w\})\setminus f^1(G^1)}\!\!\!\!\![\chi_{e'}][\chi_{e'^*}]
\nonumber \\
&=\Omega_w+\!\!\!\sum_{w'\in V_w\setminus\{w\}}\!\!\!\Big([\chi_{w'}]-\!\!\!\!\!\sum_{e\in s_E^{-1}(w')}\!\!\!\!\![\chi_e][\chi_{e^*}]\Big)
\nonumber\\
&=\Omega_w.
\end{align}
Here the last step follows from $V_w\setminus\{w\}\subseteq {\rm reg}(E)$.

We can now turn to lifting positive-length paths $\alpha$ in $E$ ending in 
\begin{equation}
(E^0\setminus f^0(G^0))\cup B_{E^0\setminus f^0(G^0)}.
\end{equation}
To begin with, if $t_{PE}(\alpha)\notin f^0(G^0)$, then $(\iota_E^0)^{-1}(\iota_E^0(t_{PE}(\alpha)))=\{t_{PE}(\alpha)\}$, so \begin{equation}
(\iota_E)^*_L([\chi_{\iota_E(\alpha)}])=[\chi_\alpha]
\end{equation}
by Lemma~\ref{targetpath}. Hence, the first spanning set of (6.3) is contained in $(\iota_E)^*_L(L_k(P))$ because 
$[\chi_{\alpha^*}]=(\iota_E)^*_L([\chi_{\iota_E(\alpha^*)})])$. 

Next, if $t_{PE}(\mu)\in B_{E^0\setminus f^0(G^0)}$, then
\begin{equation}
(\iota_E)^*_L([\chi_{\iota_E(\mu)}])=\!\!\!\!\!\!\sum_{w'\in(\iota_E^0)^{-1}(\iota_E^0(t_{PE}(\mu)))}\!\!\!\!\!\![\chi_{\nu_{w'}}]
\end{equation}
and
\begin{equation}
(\iota_E)^*_L([\chi_{\iota_E(\mu^*)}])=\!\!\!\!\!\!\sum_{w'\in(\iota_E^0)^{-1}(\iota_E^0(t_{PE}(\mu)))}\!\!\!\!\!\![\chi_{\nu_{w'}^*}]
\end{equation}
by Lemma~\ref{targetpath}, where $t_{PE}(\nu_{w'})=w'$ and $\nu_{t_{PE}(\mu)}=\mu$. Now,
\begin{equation}
\sum_{w'\in (\iota_E^0)^{-1}(\iota_E^0(t_E(\mu)))}\!\!\!\!\!\!\!\!\!\!\!\!\!\![\chi_{\nu_{w'}}]\;\;\;\Omega_{t_{PE}(\mu)}=[\chi_\mu]\;\Omega_{t_{PE}(\mu)}
\end{equation}
and 
\begin{equation}
\Omega_{t_{PE}(\mu)}\sum_{w'\in (\iota_E^0)^{-1}(\iota_E^0(t_{PE}(\mu)))}\!\!\!\!\!\!\!\!\!\!\!\!\![\chi_{\nu^*_{w'}}]=\Omega_{t_{PE}(\mu)}[\chi_{\mu^*}].
\end{equation}
Hence, the second spanning set of~\eqref{kerspan} is contained in $(\iota_E)^*_L(L_k(P))$. 

Finally, as the generators $[\chi_v]$, $v\in E^0\setminus f^0(G^0)$, and $\Omega_w$, $w\in B_{E^0\setminus f^0(G^0)}$, are elements of 
$(\iota_E)^*_L(\ker(\iota_F)^*_L)$, and kernels are ideals, we infer that all the spanning elements in (6.3) belong to $(\iota_E)^*_L(\ker(\iota_F)^*_L)$, 
which proves the desired inclusion $\ker f^*_L\subseteq (\iota_E)^*_L(\ker(\iota_F)^*_L)$. To end with, note that the last part of the theorem regarding 
unitality is immediate.
\end{proof}

\begin{theorem}\label{c*cor}
Under the assumptions  of Theorem~\ref{main}, there exists
the commutative diagram of the induced gauge-equivariant $*$-\mbox{homo}\-morphisms
\begin{equation}\label{c*pull}
\xymatrix{
&
C^*(E\underset{G}{\amalg}F)
\ar[dl]_{(\iota_E)_{C^*}^*} \ar[dr]^{(\iota_F)^*_{C^*}}&\\
C^*(E)
\ar[dr]_{f^*_{C^*}}& & 
C^*(F)
\ar[dl]^{g^*_{C^*}}\\
& C^*(G). &
}
\end{equation}
Moreover, it is a left-surjective pullback diagram in the category $\mathrm{GC^*\!A}$. Finally, if $E^0$ and $F^0$ are finite,
 then all C*-algebras in the above 
diagram and homomorphisms between them are unital, and the diagram is a pullback diagram in the category ${\rm GUC^*\!A}$.
\end{theorem}
\begin{proof}
The existence of the above commutative diagram in the category $\mathrm{GC^*\!A}$
follows from Lemma~\ref{admpush} and Corollary~\ref{corgc*}. The surjectivity of $f^*_{C^*}$ is implied by
Corollary~\ref{c*injsur}. Next,
 recall that a $*$-algebra is called AF if it is the union of a directed family of finite-dimensional $*$-subalgebras 
(e.g., see~\cite[Definition~2.2]{a-c22}). 
For any Leavitt path algebra $L_\mathbb{C}(E)$, its degree-$0$ component is a $*$-subalgebra which is AF (e.g., 
see~\cite[Proposition~2.1.14]{aasm17}). 
Therefore, using the surjectivity of $f^*_L$ (see Corollary~\ref{injsurcor}), we can apply \cite[Theorem~2.6]{a-c22} to conclude that 
\eqref{c*pull} is a pullback diagram in the category~$\mathrm{GC^*\!A}$.
 Again, the last part of the theorem regarding unitality is immediate.
\end{proof}

\section{Applications in noncommutative topology}
\noindent
The plethora of applications of Theorem~\ref{c*cor} in noncommutative topology 
goes beyond the strictly contravariant setup that is the scope
of this paper. In particular, it is very promising to combine Theorem~\ref{c*cor} with the mixed-pullback theorem
 \cite[Theorem~5.2]{ht23}. Herein, we 
will focus on
three applications: generalized Stalling's folding, collapsing line graphs to initial graphs, and projecting folding locally derived graphs onto
 their base graphs.

\subsection{Generalized foldings and multichamber even quantum spheres}
Herein, we consider the following special case of~Theorem~\ref{c*cor}:
\begin{corollary}\label{cor6.4}
Let $G$ be an admissible subgraph of $F$ and let $F$ be an admissible subgraph of~$E$. Then the following diagram 
\begin{equation}\label{pushdiag2}
\xymatrix{
&
E
&\\
F\underset{G}{\amalg}E
\ar^{\iota_g}[ur]& & 
F,
\ar_{\iota_f}[ul]\\
&
F\underset{G}{\amalg}F \ar^f[ul] \ar_{g}[ur]&
}
\end{equation}
where the map $\iota_f$ is the inclusion $F\subseteq E$ and the maps $g$, $f$, and $\iota_g$ are given by \eqref{foldeq}, \eqref{admlemfor}, and~\eqref{ig}, 
respectively, is a pushout diagram in the category of graphs. Furthermore, if we assume that $B_{E^0\setminus F^0}\subseteq G^0$, then the induced 
diagram of gauge-equivariant $*$-homomorphisms
\begin{equation}\label{c*pull2}
\xymatrix{
&
C^*(E)
\ar[dl] \ar[dr]&\\
C^*(F\underset{G}{\amalg}E)
\ar[dr] & & 
C^*(F)
\ar[dl]\\
& C^*(F\underset{G}{\amalg}F) &
}
\end{equation}
exists and is a pullback diagram in category of C*-algebras and $U(1)$-equivariant $*$-ho\-mo\-mor\-phisms.
\end{corollary}
\begin{proof}
It is straightforward to show that the diagram~\eqref{pushdiag2} is a pushout diagram in the category~{\rm OG} by verifying the universal property. To 
prove the second claim, we have to check whether the assumptions of Theorem~\ref{c*cor} are satisfied. For starters, we already know that $g$ is 
admissible, and the admissibility of $f$ is established in Proposition~\ref{admlem}. Next, the condition (P1) is immediate because $f$ is injective. Furthermore, 
note that $B_{(F\amalg_G E)^0\setminus (F\amalg_G F)^0}=B_{E^0\setminus F^0}$. Now, since $g^0$ is injective when restricted to $G^0$, the 
assumption $B_{E^0\setminus F^0}\subseteq G^0$ ensures that (P2) holds true. Finally, it also follows that $\iota_g^0(B_{(F\coprod_GE)^0\setminus(F\coprod_GF)^0})=B_{E^0\setminus F^0}$,
so (P3) holds.
\end{proof}

\begin{example}
{\rm Observe that the assumption $B_{E^0\setminus F^0}\subseteq G^0$ is not vacuous. Indeed, consider the following admissible inclusions of graphs 
$G\hookrightarrow F\hookrightarrow E$:
\begin{equation*}
\begin{tikzpicture}[auto,swap]
\tikzstyle{vertex}=[circle,fill=black,minimum size=3pt,inner sep=0pt]
\tikzstyle{edge}=[draw,->]
\tikzstyle{cycle1}=[draw,->,out=130, in=50, loop, distance=30pt]
   
\node[vertex,label=below:$v$] (0) at (0,0) {};

\path (0) edge[cycle1] node[above] {} (0);
\end{tikzpicture}\quad 
\begin{tikzpicture} 
\node (0) at (0,1) {};
\node (1) at (1,1) {};
\path (0) edge[draw,->] node[below] {${}$} (1);
\end{tikzpicture}\quad
\begin{tikzpicture}[auto,swap]
\tikzstyle{vertex}=[circle,fill=black,minimum size=3pt,inner sep=0pt]
\tikzstyle{edge}=[draw,->]
\tikzstyle{cycle1}=[draw,->,out=130, in=50, loop, distance=30pt]
   
\node[vertex, label=below:$v$] (0) at (0,0) {};
\node[vertex, label=below:$w$] (1) at (1,0) {};

\path (0) edge[cycle1] node[above] {} (0);
\path (0) edge[edge] node[left] {} (1);
\path (1) edge[cycle1] node[above] {} (1);
\end{tikzpicture}\quad
\begin{tikzpicture} 
\node (0) at (0,1) {};
\node (1) at (1,1) {};
\path (0) edge[draw,->] node[below] {${}$} (1);
\end{tikzpicture}\quad
\begin{tikzpicture}[auto,swap]
\tikzstyle{vertex}=[circle,fill=black,minimum size=3pt,inner sep=0pt]
\tikzstyle{edge}=[draw,->]
\tikzstyle{cycle1}=[draw,->,out=130, in=50, loop, distance=30pt]
   
\node[vertex, label=below:$v$] (0) at (0,0) {};
\node[vertex, label=below:$w$] (1) at (1,0) {};
\node[vertex, label=below:$z$] (2) at (2.5,0) {};

\path (0) edge[cycle1] node[above] {} (0);
\path (0) edge[edge] node[left] {} (1);
\path (1) edge[cycle1] node[above] {} (1);
\path (1) edge[edge] node[above] {${\tiny \infty}$} (2);
\end{tikzpicture}\quad.
\end{equation*}
Here the symbol $\infty$ above the arrow means that there are infinitely many edges from $w$ to $z$. It is clear that $w\in B_{E^0\setminus F^0}$ 
but $w\notin G^0=\{v\}$.}
\end{example}

Let $n,k\in\mathbb{N}$, $n\geq 1$. We inductively define the {\em multichamber sphere} $S^n_k$ via the following pushout diagram:
\begin{equation}\label{multicham}
\xymatrix{
& S_{k+1}^{n} &\\
S_k^n \ar[ur] & & B^{n}~. \ar[ul]\\
& S^{n-1} \ar@{^{(}->}[ur] \ar@{_{(}->}[ul] &
}
\end{equation}
Here $S^n_0:=B^n$ and $S^{n-1}$ is embedded in $B^n$ as the boundary sphere. Note that $S^n_1=S^n$. 
Next, consider the following pushout diagram of spaces representing the process of collapsing a chamber in a multichamber sphere ($k\geq 1$):
\begin{equation}\label{collapse}
\xymatrix{
& S_{k-1}^{n} &\\
S_k^n \ar[ur] & & B^{n}~. \ar[ul]\\
& S^{n} \ar@{>>}[ur] \ar@{_{(}->}[ul] &
}
\end{equation}
Here the map 
\begin{equation}
S^n=B^n\underset{S^{n-1}}{\amalg}B^n\longrightarrow B^n
\end{equation} 
is the flattening of a sphere defined as in \eqref{foldsphere}, and the map $S^n\hookrightarrow S^n_k$ 
is the inclusion of $S^n$ as a chamber in a multichamber sphere.

The concept of a multichamber sphere admits a straightforward generalization to the realm of noncommutative topology. Since odd quantum balls do not 
have a graph C*-algebraic presentation, in what follows we focus on the even case. 
%Let $q\in[0,1]$ and let $C(S^{2n}_q)$ be an even Hong--Szyma\'nski quantum sphere (see Example~\ref{quantsphere}). 
Now, recall from Example~\ref{quantsphere} that graphs corresponding to the C*-algebras of quantum spheres $S^n_q$ and even quantum balls 
$B^{2n}_q$ are denoted by $L_n$ and $M_{n}$, respectively.
Using Lemma~\ref{admpush}, we inductively define the graph $C_k^{2n}$ of a {\em multichamber even quantum sphere} via the following pushout diagram:
\begin{equation}
\xymatrix{
& C_{k+1}^{2n} &\\
C_k^{2n} \ar[ur] & & M_{n}~. \ar[ul]\\
& L_{2n-1} \ar@{^{(}->}[ur] \ar@{_{(}->}[ul] &
}
\end{equation}
Here $C^{2n}_0:=M_n$ and $L_{2n-1}\hookrightarrow M_n$ is the admissible inclusion of graphs corresponding to the dual boundary map 
$C(B^{2n}_q)\to C(S^{2n-1}_q)$. Note that $C^{2n}_1=L_{2n}$.

\begin{figure}[h]
\begin{center}
\begin{tikzpicture}[scale=0.4,auto,swap]
\centering
\tikzstyle{vertex}=[circle,fill=black,minimum size=3pt,inner sep=0pt]
\tikzstyle{edge}=[draw,->]
\tikzset{every loop/.style={min distance=20mm,in=130,out=50,looseness=40}}
    \node[vertex] (1) at (-2,0) {};
    \node[vertex] (2) at (0,0) {};
  %  \node[vertex] (3) at (2,0) {};
    \node[vertex] (4) at (-2,-2) {};
     \node[vertex] (5) at (2,-2) {};
     \node[vertex] (6) at (0,-2) {};
     \node[vertex] (3) at (-4,-2) {};
    \path (1) edge [draw, <-, anchor=center, loop above] node {} (1);
    \path (2) edge [draw, <-, anchor=center, loop above] node {} (2);
%    \path (3) edge [draw, <-, anchor=center, loop above] node {} (3);
    \path (1) edge [edge] node {} (2);
    \path (1) edge [edge] node {} (3);
    \path (2) edge [edge] node {} (3);
    \path (1) edge [edge] node {} (4);
    \path (1) edge [edge] node {} (5);
    \path (1) edge [edge] node {} (6);
    \path (2) edge [edge] node {} (5);
    \path (2) edge [edge] node {} (4);
    \path (2) edge [edge] node {} (6);
\end{tikzpicture}
\end{center}
\caption{The graph $C^{4}_3$.}
\end{figure}

\noindent We call $C(S^{2n}_{k,q}):=C^*(C^{2n}_k)$ the C*-algebra of the multichamber even quantum sphere $S^{2n}_{k,q}$. Next, we consider an analog of the diagram~\eqref{collapse}:
\begin{equation}
\xymatrix{
& C_{k-1}^{2n} &\\
C_k^{2n} \ar[ur] & & M_{n}~. \ar[ul]\\
& L_{2n} \ar@{>>}[ur] \ar@{_{(}->}[ul] &
}
\end{equation}
Here the graph homomorphism $L_{2n}\to M_n$ is a generalized folding \eqref{foldsphere} and $L_{2n}\to C^{2n}_k$ is the admissible inclusion mapping $L_{2n}$ to its rightmost copy inside of $C^{2n}_k$.
For instance, in the case $n=1$, we get the following pushout:
\begin{equation}
\begin{gathered}
\xymatrix{
&
\begin{tikzpicture}[scale=0.4,auto,swap]
\centering
\tikzstyle{vertex}=[circle,fill=black,minimum size=3pt,inner sep=0pt]
\tikzstyle{edge}=[draw,->]
    \node[vertex] (2) at (0,0) {};
    \node[vertex] (4) at (-2,-2) {};
    \node[vertex] (5) at (2,-2) {};
    \node (6) at (0,-2) {$\cdots$};
    \node (7) at (0,-3) {{\tiny $k$-times}};
    \path (2) edge [edge, anchor=center, loop above,min distance=20mm,in=130,out=50,looseness=40] node {} (2);
    \path (2) edge [edge] node {} (4);
    \path (2) edge [edge] node {} (5);
\end{tikzpicture}
&\\
\begin{tikzpicture}[scale=0.4,auto,swap]
\centering
\tikzstyle{vertex}=[circle,fill=black,minimum size=3pt,inner sep=0pt]
\tikzstyle{edge}=[draw,->]
    \node[vertex] (2) at (0,0) {};
    \node[vertex] (4) at (-2,-2) {};
    \node[vertex] (5) at (2,-2) {};
    \node (6) at (-0.5,-2) {$\cdots$};
    \node (7) at (0,-3) {{\tiny $(k+1)$-times}};
    \node[vertex] (8) at (1,-2) {};
    \path (2) edge [edge, anchor=center, loop above,min distance=20mm,in=130,out=50,looseness=40] node {} (2);
    \path (2) edge [edge] node {} (4);
    \path (2) edge [edge] node {} (5);
    \path (2) edge [edge] node {} (8);
\end{tikzpicture}
\ar[ur]& & 
\begin{tikzpicture}[scale=0.4,auto,swap]
\centering
\tikzstyle{vertex}=[circle,fill=black,minimum size=3pt,inner sep=0pt]
\tikzstyle{edge}=[draw,->]
    \node[vertex] (2) at (0,0) {};
    \node[vertex] (5) at (2,-2) {};
    \node (6) at (0,-3) {};
    \path (2) edge [edge, anchor=center, loop above,min distance=20mm,in=130,out=50,looseness=40] node {} (2);
    \path (2) edge [edge] node {} (5);
\end{tikzpicture}
\ar[ul]\\
&
\begin{tikzpicture}[scale=0.4,auto,swap]
\centering
\tikzstyle{vertex}=[circle,fill=black,minimum size=3pt,inner sep=0pt]
\tikzstyle{edge}=[draw,->]
    \node[vertex] (2) at (0,0) {};
    \node[vertex] (4) at (2,-2) {};
    \node[vertex] (5) at (1,-2) {};
    \path (2) edge [edge, anchor=center, loop above,min distance=20mm,in=130,out=50,looseness=40] node {} (2);
    \path (2) edge [edge] node {} (4);
    \path (2) edge [edge] node {} (5);
\end{tikzpicture}
\ar[ur]\ar[ul]&
.}
\end{gathered}
\end{equation}
Finally, taking $G:=L_{2n-1}$, $F:=M_n$, and $E:=C_{k-1}^{2n}$, we apply Corrolary~\ref{cor6.4} 
to obtain the following pullback diagram of unital gauge-equivariant $*$-homomorphisms in the category of unital $U(1)$-C*-algebras ($k\geq 1$):
\begin{equation}
\xymatrix{
& C^*(C_{k-1}^{2n}) \ar[dl] \ar[dr]&\\
C^*(C_k^{2n}) \ar[dr] & & C^*(M_{n})\ar[dl]~.\\
& C^*(L_{2n}) &
}
\end{equation}

\subsection{Line graphs and the Cuntz algebra \boldmath$\mathcal{O}_2$}
\noindent
Consider the graph $R_2$ with one vertex and two edges and the surjective admissible graph homomorphism $LR_2\to R_2$ 
(see~\eqref{dualgraphmap}):
\begin{equation}\label{cuntzfold}
\begin{gathered}
\begin{tikzpicture}[scale=0.4,auto,swap]
\centering
\tikzstyle{vertex}=[circle,fill=black,minimum size=3pt,inner sep=0pt]
\tikzstyle{edge}=[draw,->]
\tikzset{every loop/.style={min distance=20mm,in=130,out=50,looseness=40}}
    \node[vertex,label=below:$e$] (1) at (-2,0) {};
    \node[vertex,label=below:$f$] (2) at (2,0) {};
    \path (1) edge [edge, anchor=center, loop above] node[above] {$ee$} (1);
    \path (2) edge [draw,<-, anchor=center, loop above] node[above] {$ff$} (2);
    \path (1) edge [edge, bend left] node[above] {$ef$} (2);
    \path (2) edge [edge, bend left] node[below] {$fe$} (1);
\end{tikzpicture}
\begin{tikzpicture}[scale=0.4,auto,swap]
\centering
\tikzstyle{vertex}=[circle,fill=black,minimum size=3pt,inner sep=0pt]
\tikzstyle{edge}=[draw,->]
    \node (1) at (0,0) {};
    \node (2) at (0,-1) {$\longrightarrow$};
    \node (3) at (0,-2) {};    
    \node (4) at (0,-3) {};
\end{tikzpicture}
\begin{tikzpicture}[scale=0.4,auto,swap]
\centering
\tikzstyle{vertex}=[circle,fill=black,minimum size=3pt,inner sep=0pt]
\tikzstyle{edge}=[draw,->]
    \node[vertex,label=below:$v$] (1) at (0,0) {};
    \node (2) at (0,-1) {};
    \node (3) at (0,-2) {};
    \path (1) edge [edge, anchor=center, loop above, min distance=20mm, in=130, out=50, looseness=40] node[below] {$e$} (1);
    \path (1) edge [edge, anchor=center, loop below, in=130, out=50, looseness=60] node[above] {$f$} (1);
\end{tikzpicture}.
\end{gathered}
\end{equation}
Here $e,f\mapsto v$, $ee,ef\mapsto e$, and $ff,fe\mapsto f$.

Now we are ready to present an application of Theorem~\ref{c*cor} that is beyond Corollary~\ref{cor6.4} of the previous section. Consider the following pushout diagram in the category of graphs and graph homomorphisms: 
\begin{equation}
\begin{gathered}
\xymatrix{
&
\begin{tikzpicture}[scale=0.4,auto,swap]
\centering
\tikzstyle{vertex}=[circle,fill=black,minimum size=3pt,inner sep=0pt]
\tikzstyle{edge}=[draw,->]
    \node[vertex] (2) at (0,0) {};
    \node[vertex] (4) at (-2,-2) {};
    \node[vertex] (5) at (2,-2) {};
    \path (2) edge [edge, anchor=center, loop above,min distance=20mm,in=130,out=50,looseness=40] node {} (2);
    \path (2) edge [edge, anchor=center, loop above,min distance=20mm,in=130,out=50,looseness=60] node {} (2);
    \path (2) edge [edge] node {} (4);
    \path (2) edge [edge] node {} (5);
\end{tikzpicture}
&\\
\begin{tikzpicture}[scale=0.4,auto,swap]
\centering
\tikzstyle{vertex}=[circle,fill=black,minimum size=3pt,inner sep=0pt]
\tikzstyle{edge}=[draw,->]
\tikzset{every loop/.style={min distance=20mm,in=130,out=50,looseness=40}}
    \node[vertex] (1) at (-2,0) {};
    \node[vertex] (3) at (2,0) {};
    \node[vertex] (4) at (-2,-2) {};
     \node[vertex] (5) at (2,-2) {};
    \path (1) edge [edge, anchor=center, loop above] node {} (1);
    \path (3) edge [draw, <-, anchor=center, loop above] node {} (3);
    \path (1) edge [edge, bend left] node {} (3);
    \path (3) edge [edge, bend left] node {} (1);
    \path (1) edge [edge] node {} (4);
    \path (3) edge [edge] node {} (5);
\end{tikzpicture}
\ar[ur]& & 
\begin{tikzpicture}[scale=0.4,auto,swap]
\centering
\tikzstyle{vertex}=[circle,fill=black,minimum size=3pt,inner sep=0pt]
\tikzstyle{edge}=[draw,->]
    \node[vertex] (1) at (0,0) {};
    \path (1) edge [edge, anchor=center, loop above, min distance=20mm, in=130, out=50, looseness=40] node {} (1);
    \path (1) edge [edge, anchor=center, loop below, in=130, out=50, looseness=60] node {} (1);
\end{tikzpicture}
\ar[ul]\\
&
\begin{tikzpicture}[scale=0.4,auto,swap]
\centering
\tikzstyle{vertex}=[circle,fill=black,minimum size=3pt,inner sep=0pt]
\tikzstyle{edge}=[draw,->]
\tikzset{every loop/.style={min distance=20mm,in=130,out=50,looseness=40}}
    \node[vertex] (1) at (-2,0) {};
    \node[vertex] (2) at (2,0) {};
    \path (1) edge [edge, anchor=center, loop above] node {} (1);
    \path (2) edge [draw,<-, anchor=center, loop above] node {} (2);
    \path (1) edge [edge, bend left] node {} (2);
    \path (2) edge [edge, bend left] node {} (1);
\end{tikzpicture}
\ar[ur]\ar[ul]&
.}
\end{gathered}
\end{equation}
Denote the left graph and the upper graph in the above diagram by $\widetilde{LR_2}$ and $\widetilde{R_2}$, respectively. 
Here the right-bottom graph homomorphism $LR_2\to R_2$ is given 
by~\eqref{cuntzfold} and the left-bottom graph homomorphism $LR_2\to \widetilde{LR_2}$ is an admissible inclusion. 
As the assumptions of Theorem~\ref{c*cor} are clearly satisfied, we obtain the following pullback diagram in the category of unital 
$U(1)$-C*-algebras and unital gauge-equivariant $*$-homomorphims:
\begin{equation}
\xymatrix{
& C^*(\widetilde{R_2}) \ar[dl] \ar[dr]&\\
C^*(\widetilde{LR_2}) \ar[dr] & & \phantom{\mathcal{O}_2 }C^*(R_2)\cong\mathcal{O}_2 \ar[dl]~.\\
& \phantom{\cong\mathcal{O}_2 }C^*(LR_2)\cong\mathcal{O}_2 &
}
\end{equation}

\subsection{Locally derived graphs and quantum balls and spheres}
Proposition~\ref{adml} and Proposition~\ref{derprop} allow us to formulate the following:
\begin{corollary}\label{lastcor}
Let $(F,\mathcal{L})$ be a base graph such that, for all $i\in I$, the group $\Gamma_i$ is finite and the inclusion 
$F_i\hookrightarrow F$ is regular, and let  $F$ be an admissible subgraph of a graph $E$.
Then the  diagram
\begin{equation}\label{loopush}
\xymatrix{
& E &\\
E_{\mathcal{L}} \ar^{\pi_E}@{>>}[ur] & & F \ar@{_{(}->}[ul]\\
& F_\mathcal{L} \ar_{\pi_F}@{>>}[ur] \ar@{_{(}->}[ul] &
}
\end{equation}
 of admissible inclusions (the left-bottom arrow and the right-top arrow) and  projection foldings
(the right-bottom arrow and the left-top arrow)  is a pushout diagram in the category $OG$ of graphs and graph homomorphisms. 
Moreover, if 
\begin{equation}\label{breakingL}
B_{E^0\setminus F^0}\cap \bigcup_{i\in I} F_i^0=\emptyset,
\end{equation} 
the induced diagram of gauge-equivariant $*$-homomorphisms
\begin{equation*}
\xymatrix{
&
C^*(E)
\ar[dl] \ar[dr]&\\
C^*(E_{\mathcal{L}})
\ar[dr] & & 
C^*(F)
\ar[dl]\\
& C^*(F_{\mathcal{L}}) &
}
\end{equation*}
is a pullback diagram in the category of C*-algebras and $U(1)$-equivariant $*$-homomorphisms.
\end{corollary}
\begin{proof}
From the definitions of the maps involved, it is straighforward to show that the diagram~\eqref{loopush} is a~pushout diagram in the category OG by checking the universal property. It remains to check the conditions (P1)-(P3) of Theorem~\ref{main}. The condition (P1) is immediately satisfied because the inclusion map 
$F_{\mathcal{L}}\hookrightarrow E_{\mathcal{L}}$ 
is injective. The condition (P2) follows from the assumption~\eqref{breakingL}. 
Indeed, as only vertices of the form $(v,g)$ can be identified, suppose that $(v,g)\in B_{E^0_\mathcal{L}\setminus F^0_\mathcal{L}}$. For starters, since $(v,g)$ is an infinite emitter, so is $v$ is by the finiteness of all $\Gamma_i$. Better still, $v$ emits infinitely many edges beyond $F^0$.
Furthermore, if $v$ would emit infinitely many or zero edges into $F^0$, then $(v,g)$ would emit infinitely many or zero edges into~$F_\mathcal{L}^0$, which contradicts the assumption that $(v,g)$ is a breaking vertex. We thus conclude that $v\in B_{E^0\setminus F^0} \cap \bigcup_{i\in I}F_i^0$, which contradicts~(7.7). Hence,
an element of $B_{E_\mathcal{L}^0\setminus F_\mathcal{L}^0}$ cannot be of the form $(v,g)$, which proves~(P2).
Finally, let $v\in B_{E_\mathcal{L}^0\setminus F_\mathcal{L}^0}$. Then $\pi_F^0(v)=\pi_E^0(v)$ is an infinite emitter in $E$
because $\pi_E^1$ is proper. On the other hand, $v\in\mathrm{reg}(F_\mathcal{L})$, $(\pi_F^0)^{-1}(\pi_F^0(v))=\{v\}$
and $\pi_F^1$ is surjective, so $\pi_F^0(v)\in\mathrm{reg}(F)$. Therefore, $\pi_E^0(v)\in B_{E^0\setminus F^0}$,
 which proves~(P3).
\end{proof}

We end the paper by exemplifying Corollary~\ref{lastcor}. We begin with a very simple example
\begin{equation}
\begin{gathered}
\xymatrix{
&
\begin{tikzpicture}[scale=0.4,auto,swap]
\centering
\tikzstyle{vertex}=[circle,fill=black,minimum size=3pt,inner sep=0pt]
\tikzstyle{edge}=[draw,->]
\tikzset{every loop/.style={min distance=20mm,in=130,out=50,looseness=40}}
    \node[vertex] (1) at (-4,0) {};
    \node[vertex] (2) at (0,0) {};
    \path (1) edge [edge, anchor=center, loop above] node {} (1);
    \path (1) edge [edge] node {} (2);
\end{tikzpicture}
&\\
\begin{tikzpicture}[scale=0.4,auto,swap]
\centering
\tikzstyle{vertex}=[circle,fill=black,minimum size=3pt,inner sep=0pt]
\tikzstyle{edge}=[draw,->]
\node (0) at (5,0) {};
    \node[vertex] (1) at (0,0) {};
    \node[vertex] (2) at (1,-1) {};
    \node[vertex] (3) at (2,-1) {};
    \node[vertex] (4) at (3,0) {};
    \node[vertex] (5) at (3,1) {};
    \node[vertex] (6) at (2,2) {};
    \node[vertex] (7) at (1,2) {};
    \node[rotate=60] (8) at (0.5,1) {{\tiny $\cdots$}};
 %   \node (9) at (-2,1) {{\tiny $n$ edges}};
    \node[vertex] (10) at (6,-1) {};
    \path (1) edge[edge] node {} (2);
    \path (2) edge[edge] node {} (3);
    \path (3) edge[edge] node {} (4);
    \path (4) edge[edge] node {} (5);
    \path (5) edge[edge] node {} (6);
    \path (6) edge[edge] node {} (7);
    \path (3) edge[edge] node {} (10);
\end{tikzpicture}
\ar[ur]& & 
\begin{tikzpicture}[scale=0.4,auto,swap]
\centering
\tikzstyle{vertex}=[circle,fill=black,minimum size=3pt,inner sep=0pt]
\tikzstyle{edge}=[draw,->]
\tikzset{every loop/.style={min distance=20mm,in=130,out=50,looseness=40}}
    \node[vertex] (1) at (-2,0) {};
    \path (1) edge [edge, anchor=center, loop above] node {} (1);
\end{tikzpicture}
\ar[ul]\\
&
\begin{tikzpicture}[scale=0.4,auto,swap]
\centering
\tikzstyle{vertex}=[circle,fill=black,minimum size=3pt,inner sep=0pt]
\tikzstyle{edge}=[draw,->]
\node (0) at (5,0) {};
    \node[vertex] (1) at (2,0) {};
    \node[vertex] (2) at (3,-1) {};
    \node[vertex] (3) at (4,-1) {};
    \node[vertex] (4) at (5,0) {};
    \node[vertex] (5) at (5,1) {};
    \node[vertex] (6) at (4,2) {};
    \node[vertex] (7) at (3,2) {};
    \node[rotate=60] (8) at (2.5,1) {{\tiny $\cdots$}};
    \node (9) at (1.5,1) {};
    \path (1) edge[edge] node {} (2);
    \path (2) edge[edge] node {} (3);
    \path (3) edge[edge] node {} (4);
    \path (4) edge[edge] node {} (5);
    \path (5) edge[edge] node {} (6);
    \path (6) edge[edge] node {} (7);
\end{tikzpicture}
\ar[ur]\ar[ul]&
}
\end{gathered}
\end{equation}
of a pushout diagram of graphs in the category OG that induces a pullback diagram in the category of unital $U(1)$-C*-algebras.
The right-top $*$-homomorphism of this pullback diagram is the symbol map $\sigma$ giving rise to the standard
 short exact sequence for the Toeplitz C*-algebra~$\mathcal{T}$:
\begin{equation}
0\longrightarrow\mathcal{K}\longrightarrow\mathcal{T}\stackrel{\sigma}{\longrightarrow}C(S^1)\longrightarrow 0.
\end{equation}
Here $\mathcal{K}$ stands for the C*-algebra of compact operators. Tensoring this short exact sequence with the matrix algebra
$M_n(\mathbb{C})$, we obtain
\begin{equation}
0\longrightarrow\mathcal{K}\longrightarrow\mathcal{T}\otimes M_n(\mathbb{C})\stackrel{}{\longrightarrow}C(S^1)
\otimes M_n(\mathbb{C})\longrightarrow 0.
\end{equation}
Now, possibly except for the middle term, the above short exact sequence coincides with the short exact sequence obtained
from the left-bottom $*$-homomorphism of the induced pullback diagram.

To go beyond the simple setting of the above example, 
consider the graphs $M_2$ and $L_3$ of the  C*-algebra $C(B^4_q)$ of the even quantum ball $B^4_q$ and 
the C*-algebra $C(S^3_q)$ of the boundary quantum sphere~$S^3_q$,
 respectively. Recall that
$L_3^0:=\{v_1,v_2\}$, $L^1_3:=\{e_{11},e_{12},e_{22}\}$, and $s_{L_3}(e_{ij}):=v_i\,$, $t_{L_3}(e_{ij}):=v_j\,$, $1\leq i\leq j\leq 2$.
Next, let $\Gamma_1:=\mathbb{Z}/3\mathbb{Z}$, $\Gamma_2:=\mathbb{Z}/2\mathbb{Z}$, and 
let $F^0_i:=\{v_i\}$, $F^1_i:=\{e_{ii}\}$, define  subgraphs $F_i\subseteq L_3$, $i=1,2$. 
Finally, put $\gamma_n$ for the generator of $\mathbb{Z}/n\mathbb{Z}$ and define
\begin{equation}
\mathcal{L}:t_{L_3}^{-1}(v_1)\cup t_{L_3}^{-1}(v_2)\longrightarrow\mathbb{Z}/3\mathbb{Z}\sqcup\mathbb{Z}/2\mathbb{Z},\quad 
e_{11}\longmapsto \gamma_3,
\quad e_{12}\longmapsto \gamma_2,\quad e_{22}\longmapsto \gamma_2\,.
\end{equation}
Now the pushout of Corollary~\ref{lastcor} takes the form
\begin{equation}
\begin{gathered}
\xymatrix{
&
\begin{tikzpicture}[scale=0.4,auto,swap]
\centering
\tikzstyle{vertex}=[circle,fill=black,minimum size=3pt,inner sep=0pt]
\tikzstyle{edge}=[draw,->]
\tikzset{every loop/.style={min distance=20mm,in=130,out=50,looseness=40}}
    \node[vertex] (1) at (-4,0) {};
    \node[vertex] (2) at (0,0) {};
    \node[vertex] (3) at (4,0) {};
    \path (1) edge [edge, anchor=center, loop above] node {} (1);
    \path (2) edge [draw,<-, anchor=center, loop above] node {} (2);
    \path (1) edge [edge] node {} (2);
    \path (1) edge [edge, bend right] node {} (3);
    \path (2) edge [edge] node {} (3);
\end{tikzpicture}
&\\
\begin{tikzpicture}[scale=0.4,auto,swap]
\centering
\tikzstyle{vertex}=[circle,fill=black,minimum size=3pt,inner sep=0pt]
\tikzstyle{edge}=[draw,->]
    \node[vertex] (1) at (-2,0) {};
    \node[vertex] (2) at (2,0) {};
    \node[vertex] (11) at (-3,2) {};
    \node[vertex] (12) at (-1,2) {};
    \node[vertex] (21) at (2,2) {};
    \node[vertex] (3) at (6,0) {};
    \path (1) edge [edge,bend left] node {} (11);
    \path (11) edge [edge, bend left] node {} (12);
    \path (12) edge [edge, bend left] node {} (1);
    \path (1) edge [edge] node {} (2);
    \path (1) edge [edge] node {} (21);
    \path (2) edge [edge, bend left] node {} (21);
    \path (21) edge [edge, bend left] node {} (2);
    \path (2) edge [edge] node {} (3);
    \path (1) edge [edge, bend right] node {} (3);
\end{tikzpicture}
\ar[ur]& & 
\begin{tikzpicture}[scale=0.4,auto,swap]
\centering
\tikzstyle{vertex}=[circle,fill=black,minimum size=3pt,inner sep=0pt]
\tikzstyle{edge}=[draw,->]
\tikzset{every loop/.style={min distance=20mm,in=130,out=50,looseness=40}}
    \node[vertex] (1) at (-2,0) {};
    \node[vertex] (2) at (2,0) {};
    \path (1) edge [edge, anchor=center, loop above] node {} (1);
    \path (2) edge [draw,<-, anchor=center, loop above] node {} (2);
    \path (1) edge [edge] node {} (2);
\end{tikzpicture}
\ar[ul]\\
&
\begin{tikzpicture}[scale=0.4,auto,swap]
\centering
\tikzstyle{vertex}=[circle,fill=black,minimum size=3pt,inner sep=0pt]
\tikzstyle{edge}=[draw,->]
    \node[vertex] (1) at (-2,0) {};
    \node[vertex] (2) at (2,0) {};
    \node[vertex] (11) at (-3,2) {};
    \node[vertex] (12) at (-1,2) {};
    \node[vertex] (21) at (2,2) {};
    \path (1) edge [edge,bend left] node {} (11);
    \path (11) edge [edge, bend left] node {} (12);
    \path (12) edge [edge, bend left] node {} (1);
    \path (1) edge [edge] node {} (2);
    \path (1) edge [edge] node {} (21);
    \path (2) edge [edge, bend left] node {} (21);
    \path (21) edge [edge, bend left] node {} (2);
\end{tikzpicture}
\ar[ur]\ar[ul]&
}
\end{gathered}
\end{equation}
and gives rise to the pullback diagram in the category of unital $U(1)$-C*-algebras and unital $U(1)$-equivariant $*$-homomorphisms
\begin{equation}
\xymatrix{
&
C(B^4_q)
\ar[dl] \ar[dr]&\\
C^*((M_2)_{\mathcal{L}})
\ar[dr] & & 
C(S^3_q).
\ar[dl]\\
& C^*((L_3)_{\mathcal{L}}) &
}
\end{equation}
It is straightforward to generalize the above example to any even quantum ball and its boundary quantum sphere and any finite cyclic groups.

\section*{Acknowledgements} 
\noindent
We would like to thank Alexander Frei for numerous conversations about graph C*-algebras and Cuntz--Pimsner algebras, 
Roozbeh Hazrat for drawing our attention to the isomorphism $C^*(LR_2)\cong\mathcal{O}_2$,  Jack Spielberg
for showing us the isomorphism $C^*((R_2)_\mathcal{L})\cong M_2(\mathcal{O}_3)$, 
and the referee for numerous very useful suggestions. We hereby acknowledge
that research visits to Warsaw of the second author and Roozbeh Hazrat, which vastly facilitated finishing this paper,
 were partially financed by the University of Warsaw Thematic Research Programme \emph{Quantum
Symmetries}.
This work is part of the project ``Applications of graph algebras and higher-rank graph algebras
in noncommutative geometry'' partially supported by NCN grant UMO-2021/41/B/ST1/03387.

\end{document}